\documentclass[a4paper]{amsart}

\usepackage[utf8]{inputenc}
\DeclareUnicodeCharacter{00A0}{ }
\usepackage{amssymb}
\usepackage{MnSymbol}
\usepackage{enumitem}

\usepackage{ae}
\usepackage[utf8]{inputenc}
\numberwithin{figure}{section}

\usepackage{mathrsfs,a4wide}
\usepackage{esint}
\usepackage{color}
\usepackage[mathscr]{euscript}

\usepackage{import}

\usepackage[leqno]{amsmath}

\numberwithin{figure}{section}

\newtheorem{theorem}{Theorem}[section]
\newtheorem*{mtheorem}{Main Theorem}
\newtheorem{lemma}[theorem]{Lemma}
\newtheorem{proposition}[theorem]{Proposition}

\theoremstyle{definition}
\newtheorem{definition}[theorem]{Definition}

\newtheorem{remark}[theorem]{Remark}
\numberwithin{equation}{section}

\newcommand{\R}{\mathbb{R}}
\newcommand{\N}{\mathbb{N}}

\newcommand{\Ha}{\mathcal{H}}

\newcommand{\beq}{\begin{equation}}
\newcommand{\eeq}{\end{equation}}

\newcommand{\eps}{\varepsilon}
 
\newcommand{\la}{\langle}
\newcommand{\ra}{\rangle}
\newcommand{\Div}{\operatorname{div}}
\newcommand{\pa}{\partial}

\newcommand{\medint}{-\kern -,375cm\int}
\newcommand{\medintinrigo}{-\kern -,315cm\int}
\newcommand{\e}{\varepsilon}

\begin{document}

\title[Nonlocal mean curvature flow]{Short time existence of the classical solution to the fractional mean curvature flow}

\author{Vesa Julin}

\author{Domenico Angelo La Manna}

\keywords{}

\begin{abstract} 
We establish short-time existence of the smooth solution to the fractional mean curvature flow  when the initial set is bounded and  $C^{1,1}$-regular. We provide  the same result also  
for the volume preserving     fractional mean curvature flow.
\end{abstract}

\maketitle

\section{Introduction}

In this paper we study the motion of  sets by their  fractional  mean curvature from the classical point of view. In the problem  we are given a bounded regular  set $E_0 \subset \R^{n+1}$ and we 
let it  evolve to a family of smooth sets $(E_t)_{t \in (0,T]}$ according to the law where the normal velocity   at a point equals its  fractional mean curvature. More precisely this can be written as 
\begin{equation}
\label{flow}
V_t = -H_{E_t}^s \qquad \text{on } \, \pa E_t ,
\end{equation}
where $V_t$ denotes the normal velocity and $H_{E}^s$ is the  fractional mean  curvature of $E \subset \R^{n+1}$ with $s \in (0,1)$, given by 
\beq \label{def FMC}
H_{E}^s(x) := \text{p.v.} \,  \left( \int_{E^c} \frac{dy}{|y-x|^{n+1+s}}  - \int_{E} \frac{dy}{|x -y|^{n+1+s}} \right).
\eeq

One  motivation  to study \eqref{flow}  is   that it can be interpreted as the gradient flow of the fractional perimeter and  is therefore the evolutionary counterpart to the 
problem of minimizing the fractional perimeter. The stationary problem has received a lot of attention since the  work \cite{CRS}, where the authors study the regularity of 
 sets which locally minimize the fractional perimeter. By   \cite{CRS} and  \cite{BFV} we know that the local   perimeter minimizers are smooth up to  a small singular set which precise  dimension  is not known (see \cite{SV} and \cite{DPW} and the references therein).  The global isoperimetric problem in the Euclidian space is also well understood.   It is proven in \cite{FS} that  the ball is the solution to the fractional isoperimetric problem, and we even  have the sharp   quantification of the isoperimetric inequality  \cite{FFMMM, FMM}.  A related result is the generalization of the Alexandroff theorem \cite{CFSW, CFMN}, where the authors prove that the  ball is the only smooth compact set with constant fractional mean curvature.  Note that one does not need to assume the  set  to be connected, which is in contrast to the classical Alexandroff theorem.

Concerning  the evolutionary  problem, the first existence result  is due to  Imbert \cite{imb} who defines the \emph{viscosity solution} of \eqref{flow} and proves that it  exists for all times 
 and is unique.  This means that  the flow \eqref{flow} has a well defined weak solution which is unique up to fattening. In  \cite{CS} Caffarelli-Souganidis construct  weak solution by using threshold dynamics,  and   in \cite{CMP}   Chambolle-Morini-Ponsiglione use the  gradient flow structure   to construct weak solution by using the  minimizing movement scheme (see also \cite{CMP1}). 

The issue with  the weak solution is,  as observed in \cite{CSV1} (see also \cite{CDNV}),  that the flow may develop singularities even in the planar case. Cinti-Sinestrari-Valdinoci  \cite{CSV2} avoid singularities 
 by studying  the volume preserving fractional mean curvature flow
\begin{equation}
\label{flow vol}
V_t = -(H_{E_t}^s - \bar{H}_{E_t}^s) \qquad \text{on } \, \pa E_t ,
\end{equation}
 for convex initial sets, where  $\bar{H}_{E_t}^s = \fint_{\pa  E_t} H_{E_t}^s \, d \Ha^n$. By \cite{CNR}, the fractional mean curvature flow with a forcing term preserves convexity and thus  one  expects Huisken's result \cite{Hui2} to hold also in the fractional  case, i.e., the flow \eqref{flow vol} remains convex, does not develop singularities  and converges to the sphere. This is indeed the main result in  \cite{CSV2} under an additional regularity assumption. We remark that  also \eqref{flow} preserves the convexity and therefore one may expect a result similar to  \cite{Hui1}  to hold also for \eqref{flow}, but to the best of our knowledge no result in this direction exists. Finally we refer to \cite{SaeV} for interesting analysis of the smooth solution of \eqref{flow}.

Here we are interested in the classical solution of \eqref{flow} which means that the sets  $(E_t)_t$ are smooth and  diffeomorphic  to $E_0$. As we mentioned  we may expect 
the classical solution to exist only  for a short time interval $(0,T)$ as the flow will vanish in finite time or it  may develop singularities before that.  We prove the short time existence of the classical solution under the assumption that the initial set is $C^{1,1}$-regular, or $C^{1+s+\alpha}$-close to a $C^{1,1}$-regular set.      The same result holds also for the volume preserving flow  \eqref{flow vol} and  we choose to state the result for both cases simultaneously. 

\begin{mtheorem} \label{main thm}
Let $E_0 \subset \R^{n+1}$ be a bounded set such that $\pa E_0$ is a $C^{1,1}$-regular hypersurface and let $\alpha \in (0,1)$. There exists $T \in (0,1)$ such that the fractional mean curvature flow \eqref{flow}, and the volume preserving flow \eqref{flow vol}, has  a unique classical solution  $(E_t)_{t \in (0,T]}$ starting from $E_0$. The flow becomes instantaneously smooth, i.e.,  each surface $\pa E_t$  with $t \in (0,T]$ is $C^\infty$-hypersurface. Moreover, there is $\eps >0$ with the property that if the $C^{1+s+\alpha}$-distance between $E_0$ and $E$ is less than $\eps$, then the flow \eqref{flow} (and the flow \eqref{flow vol}) starting from $E$ also exists for the time interval $(0,T]$. 
\end{mtheorem}

We give a more quantitative statement of the main theorem in the last section in Theorem \ref{main thm num} and  in Theorem \ref{main thm num 2}.  We expect  the smooth solution of \eqref{flow}   to agree with the viscosity solution on the time interval $(0,T]$ but  we do not prove this.

The proof of the main theorem is based on Schauder estimates on parabolic equations. As in \cite{HP} we first  parametrize the flow \eqref{flow} by using the 'height' function over a smooth reference surface $\Sigma$, which is close to the initial surface. This leads us to solve the nonlinear 
nonlocal parabolic equation 
\beq \label{the EQ}
\pa_t u = \Delta^{\frac{1+s}{2}}u +  Q(x,u, \nabla u) - H_\Sigma^s\qquad \text{on } \, \Sigma \times (0,T],
\eeq
where $\Delta^{\frac{1+s}{2}}$ denotes the fractional Laplacian on $\Sigma$ and $Q$ is nonlinear.  Following the idea in \cite{ES2}  we prove that  the nonlinear term $Q$ in \eqref{the EQ} remains small 
 for a short time interval and the equation can thus be seen as a small perturbation of the fractional heat equation. We then  use Schauder estimates and a standard fixed point argument to obtain the existence of a solution which is $C^{1+s+\alpha}$-regular in space and  obtain the  smoothness by differentiating the equation.  The right hand side  of \eqref{the EQ} is essentially the parametrization of the fractional mean curvature.  Similarly as  in \cite{BFV} and \cite{DPW},  the main difficulty  in our analysis is due to  the complicated structure of  the nonlinear term $Q$, which makes it  challenging to estimate its $C^{k+\alpha}$-norm in a quantitative way.  In order to differentiate the equation \eqref{the EQ} multiple times we need effective notation and basic tools from    differential geometry.  We also point out that the $C^{1,1}$-regularity of the intial set causes  additional difficulties, because some of the constants in \eqref{the EQ} depend on the $C^{\alpha}$-norm of the curvature of $\Sigma$, which we cannot bound uniformly if we want $\Sigma$ to be close to $\pa E_0$. Finally another technical issue, although a minor one,  is that there is no  comprehensive Schauder theory for the  fractional heat equation  on  compact hypersurfaces   in the literature,   and therefore we have to prove the  appropriate estimates ourselves (see Theorem \ref{parabolic est}).   

The paper is organized as  follows. After a short preliminary section (Section 2) we derive the equation \eqref{the EQ} in Section 3  by using the parametrization by the height function (Proposition \ref{final equation}). In  Section 4  we study the spatial regularity properties of the  operator on the right-hand-side of \eqref{the EQ} and give the proof of the main theorem  in Section 5.  The Appendix contains the Schauder estimates for the  fractional heat equation with a forcing term on  compact hypersurfaces, which might be of independent interest.

\section{Preliminaries}

Throughout the paper we assume that $\Sigma \subset \R^{n+1}$ is a  smooth compact hypersurface, i.e., there is a smooth bounded set 
$G \subset  \R^{n+1}$ such that $\pa G  = \Sigma$. We choose $G$ as our reference set and  define the classical solution of \eqref{flow} (and \eqref{flow vol})  as in the case of the classical mean curvature flow \cite{MantegazzaBook}, i.e., we say that $(E_t)_{t \in (0,T]}$ is a classical solution of \eqref{flow} starting from $E_0$ if there exists a map $\Psi \in   C(\R^{n+1} \times [0,T]; \R^{n+1}) \cap C^\infty(\R^{n+1} \times (0,T]; \R^{n+1}) $ such that  $\Psi(\cdot, t)$ is a $C^{1+s}$-diffeomorphism for $t \in [0,T]$ and  smooth diffeomorphism for $t \in (0,T]$, with   $E_t = \Psi(G, t)$ for  $t \in (0,T]$, $E_0 = \Psi(G, 0)$ and  $E_t$ satisfies  \eqref{flow} (or \eqref{flow vol} in the volume preserving case).

\subsection{Geometric preliminaries}

We recall some basic  analysis  related to Riemannian manifolds. For an introduction to the topic we refer to \cite{Lee}, from where we also adopt our notation.  

Since $\Sigma$ is  embedded in $\R^{n+1}$  it has a natural metric $g$ induced by the Euclidian metric. Then $(\Sigma, g)$ is a Riemannian manifold and we denote the inner product on each tangent space $X, Y \in T_x \Sigma$ by $\la X, Y \ra$. We extend the inner product in a natural way for tensors. We denote by $\mathscr{T}(\Sigma)$ the smooth vector fields on $\Sigma$ and recall that for  $X \in \mathscr{T}(\Sigma)$ and $u \in C^\infty(\Sigma)$ the notation 
$Xu$ means the derivative of $u$ in direction of $X$. We emphasize that we assume every vector field to be smooth.

 We denote the Riemannian connection on  $\Sigma$ by $\nabla$ and  recall that  for a function $u \in C^\infty(\Sigma)$ the covariant derivative  $\nabla u $ is a $1$-tensor field defined for  $X  \in \mathscr{T}( \Sigma)$  as
\[
\nabla u(X)  = \nabla_X u = X u,
\]
i.e., the derivative  of $u$ in the direction of $X$.  The  covariant derivative  of  a smooth $k$-tensor field $F \in \mathscr{T}^k( \Sigma)$, denoted  by  $\nabla F$, is a $(k+1)$-tensor field    and is   defined for $ Y_1, \dots, Y_k, X \in \mathscr{T}( \Sigma)$ as 
\[
\nabla F(Y_1, \dots, Y_k, X) = (\nabla_X F)(Y_1, \dots, Y_k) ,
\]
where
\[
(\nabla_X F)(Y_1, \dots, Y_k) = X F(Y_1, \dots, Y_k) - \sum_{i=1}^k F(Y_1, \dots,  \nabla_X Y_i ,\dots, Y_k).
\]
Here $\nabla_X Y$ is  the covariant derivative of $Y$ in the direction of $X$ (see \cite{Lee}) and since $\nabla$ is the Riemannian connection it holds  $\nabla_X Y = \nabla_Y X  + [X,Y]$ for every $X, Y \in \mathscr{T}( \Sigma)$.

We denote the $k$th order  covariant derivative of a smooth function  $u \in C^\infty(\Sigma)$ by $\nabla^k u$, which   is a $k$-tensor field defined recursively as $\nabla^k u = \nabla (\nabla^{k-1} u)$.  Let  $X_1, \dots, X_k \in \mathscr{T}( \Sigma)$ be vector fields on $\Sigma$. Then $\nabla^k u(X_1, \dots, X_k)$ denotes the covariant derivative applied to   $X_1, \dots, X_k$ and we often  use the  notation 
\[
\nabla_{X_k} \cdots \nabla_{X_1} u = \nabla^k u(X_1, \dots, X_k).
\]

We use the fact that  $\Sigma$ is embedded in $\R^{n+1}$ and define the sup-norm and  the $\alpha$-H\"older norm, for $\alpha \in (0,1)$,  of a function $u \in C(\Sigma)$ in a standard way,
\[
\|u\|_{C^0(\Sigma)} :=  \sup_{x \in \Sigma} |u(x)|
\]
and
\[
\|u\|_{C^\alpha(\Sigma)}  := \sup_{\substack{x, y \in \Sigma \\ x \neq y}}  \frac{|u(y) - u(x)|}{|y-x|^\alpha} +  \|u\|_{C^0(\Sigma)}. 
\]
Moreover, we set $\|u\|_{C^k(\Sigma)}  := \sum_{l =0}^k  \|\nabla^l u\|_{C^0(\Sigma)}$   for all $k \in \N$. We define further  the $\alpha$-H\"older norm  of a $k$-tensor field $F \in \mathscr{T}^k( \Sigma)$ by 
\[
\|F\|_{C^\alpha(\Sigma)} := \sup \{  \|F(X_1, \dots, X_k)\|_{C^\alpha(\Sigma)} : X_i   \in \mathscr{T}( \Sigma)  \,\, \text{with}\,\, \| X_i \|_{C^1(\Sigma)} \leq 1 , \,\, i = 1 \dots, k \} .
\]
It is straightforward to check that this agrees with the more standard definition via partition of unity.  Using this we define the $C^{k + \alpha}$- norm of function $u$,  with $k \in \mathbb{N}$ and $ \alpha\in [0,1) $, as
\[
\|u\|_{C^{k+\alpha}(\Sigma)} : =  \sum_{l=0}^k \| \nabla^l u\|_{C^{\alpha}(\Sigma)}.
\]
We use the notation $u \in C^{k+\alpha}(\Sigma)$ for a function $u$ with bounded  $C^{k+\alpha}(\Sigma)$-norm when $\alpha \in (0,1)$  and $u \in C^{k,1}(\Sigma)$ when $\alpha = 1$. We assume that $\Sigma = \pa G$ is uniformly $C^{1,1}$-regular surface and define its $C^{1,1}$-norm as the smallest number $R$ such that $G$ satisfies the interior and the exterior ball condition with radius $1/R$.  Note that this norm also  bounds $ \|\nu\|_{C^{1}(\Sigma)}$.  If a constant depends on the $C^{1,1}$-norm of $\Sigma$ we choose not  to mention it and call such a constant \emph{uniform}.

We recall the following  interpolation inequality.  The proof is essentially the same as \cite[Lemma 6.32]{GT} (see also \cite{AubinBook2}).
\begin{lemma}
\label{aubinlemma} 
Assume $\Sigma$ is a compact $C^{1,1}$-hypersurface and let $s \in (0,1)$ and   $\alpha \in (0,1-s)$. For every  $\delta \in (0,1)$ there is $C_\delta>0$ such that for  $u \in C^{1+s+\alpha}(\Sigma)$ it holds
\[
\|u\|_{C^{1 +\alpha}(\Sigma)} \leq \delta \|u\|_{C^{1+s+\alpha}(\Sigma)} + C_\delta \|u\|_{C^{0}(\Sigma)}.
\]
\end{lemma}

Observe  that  $\nabla^2 u$ is  symmetric, i.e., $\nabla^2 u (X, Y) = \nabla^2 u (Y, X)$ for every vector fields $X$ and $Y$, while $\nabla^k u$ for $k \geq 3$ is not. From the definition we see that for  $X_1, \dots, X_k \in \mathscr{T}( \Sigma)$ with $\|X_i\|_{C^{k+2}(\Sigma)} \leq 1$, $i = 1, \dots, k$  it holds
\beq \label{vaihto kikka}
\nabla^k u (X_1, \dots,  X_i ,\dots, X_j, \dots, Y_k)  =  \nabla^k  u(X_1, \dots,  X_j ,\dots, X_i, \dots, Y_k) + \pa^{k-1} u.
\eeq
Here the notation $\pa^{l} u$ stands for  a function which satisfies 
\beq \label{pa nota}
\|\pa^{l} u\|_{C^{\gamma}(\Sigma)} \leq C_{l,\gamma} \|u\|_{C^{l+\gamma}(\Sigma)} 
\eeq
 for any $\gamma \in (0,2)$. Note also  that  in general  $\nabla_Y \nabla_X u \neq \nabla_Y (\nabla_X u )$ for  $X, Y \in \mathscr{T}(\Sigma)$ since  $\nabla_Y (\nabla_X u ) =Y(X u) = YX u$ and $\nabla_Y \nabla_X u = YX u- (\nabla_{Y}X) u $.  On the other hand if $X_1, \dots, X_k$ are vector fields  with $\|X_i\|_{C^{k+2}(\Sigma)} \leq 1$, $i = 1, \dots, k$,  then it holds
\beq \label{derivointi kikka}
\nabla_{X_k}  \cdots \nabla_{X_1} u = X_k \cdots X_1 u + \pa^{k-1} u,
\eeq
where $\pa^{k-1} u$ denotes a function which satisfies \eqref{pa nota}. It is then straightforward to check  that  for any $\gamma \in (0,2)$  it holds
\beq \label{normi kikka}
\begin{split}
\sup \{  \| X_k \cdots X_1  u\|_{C^\gamma(\Sigma)} &:  X_i \in \mathscr{T}( \Sigma)  \, ,\,  \|X_i\|_{C^{k+2}(\Sigma)} \leq 1 \,, \, i = 1, \dots, k\} \\
&\geq   \frac{1}{C_k} \|u\|_{C^{k +\gamma}(\Sigma)} - C_k \|u\|_{C^{k-1 +\gamma}(\Sigma)} ,
\end{split}
\eeq
where  $C_k$ depends on $k$.

 We may use the fact that $\Sigma$ is embedded in $\R^{n+1}$  to extend any function $F \in C^1(\Sigma; \R^m)$ to  $\tilde F \in C^1(\R^{n+1}; \R^m)$ such that $\tilde F = F$ on $\Sigma$.  We  define  the \emph{tangential differential}  of $F$ by 
\[
D_\tau F(x) = D \tilde F(x) (I - \nu(x)\otimes \nu(x)) ,
\]
where $\nu$  denotes the unit outer normal of $\Sigma = \pa G$ (outer with respect to $G$). We denote $\nu_\Sigma$ if we want to be emphasize that the normal is related to $\Sigma$ and  denote by  $\nu_E(x)$  the  normal of a generic set $E$.  It is clear that $D_\tau F(x)$ does not depend on the chosen extension. With a slight abuse of notation we denote the tangential gradient of $u \in C^\infty(\Sigma)$ at $x $ also by $D_\tau u(x)$, even if it is a vector in $\R^{n+1}$. 

We may use the embedding to  associate the tangent space  $T_x \Sigma$ with the linear subspace $\{ p \in \R^{n+1} :  p \cdot \nu(x) = 0\} $ by the relation 
\[
v(u) = D_\tau u(x) \cdot p \qquad \text{for all } \, u \in C^\infty(\Sigma),
\]
where $v \in T_x \Sigma$, i.e., a derivation at $x$,  and $p \in \R^{n+1}$ with $p \cdot \nu(x) = 0$. The components of the vector $p$ are then given by $p_i = v(x_i)$.  Indeed, by  'tangent space' we usually mean the geometric tangent space, i.e.,   a linear subspace of $\R^{n+1}$, but for clarity we use  '$\cdot$' for the standard inner product of two vectors in $\R^{n+1}$ while '$\la \cdot, \cdot \ra$' denotes the inner product on the tangent space. Similarly we may associate a smooth  vector field $X \in \mathscr{T}(\Sigma)$ with the vector valued function  $\tilde{X} \in C^\infty(\Sigma, \R^{n+1})$ which satisfies $\tilde{X}(x) \cdot \nu(x) = 0$ for all $x \in \Sigma$ and 
\[
\nabla_X u = D_\tau u \cdot \tilde{X} \qquad \text{for all } \, u \in C^\infty(\Sigma).
\]
Therefore, by a vector field $X$ we usually mean a vector valued function which values are on the (geometric) tangent space, $X \cdot \nu = 0$, with the convention that $X u $ denotes the derivative of $u$ in direction of $X$.   It is also clear that the tangential gradient of $u \in C^\infty(\Sigma)$ is equivalent to its covariant derivative and  for every $\alpha \in (0,1)$ it holds
\[
\frac{1}{C}\|u\|_{C^{1+\alpha}(\Sigma)} \leq \|D_\tau u\|_{C^{\alpha}(\Sigma)} + \|u\|_{C^0(\Sigma)} \leq C \|u\|_{C^{1+\alpha}(\Sigma)}.
\]

We denote the divergence of a vector field  $X \in  \mathscr{T}( \Sigma)$  by  $\Div X$ and the divergence theorem states 
\[
\int_{\Sigma} \Div X \,  d\Ha^{n} = 0. 
\]
For clarity we denote  the divergence of a vector valued function $F  \in C^\infty( \R^{n+1},  \R^{n+1})$  in $\R^{n+1}$ by $\Div_{\R^{n+1}}F$.  We may extend the definition of divergence to vector valued functions  $\tilde X \in C^\infty(\Sigma, \R^{n+1})$ by $\Div \tilde{X} := \text{Trace}(D_\tau \tilde{X})$. Then the divergence theorem generalizes  to 
\[
\int_{\Sigma} \Div \tilde{X} \,  d\Ha^{n} = \int_{\Sigma} H_\Sigma \,  \tilde{X}\cdot \nu  \,  d\Ha^{n},
\]
where $H_\Sigma$ denotes the mean curvature of $\Sigma$, which   is the sum of the principal curvatures.

\subsection{Fractional Laplacian}

We define the fractional Laplacian on $\Sigma$ as
\[
\Delta^{\frac{s+1}{2}} u(x) := 2  \int_{\Sigma}\frac{u(y)- u(x)}{|y-x|^{n+1+ s}} \, d \Ha_y^{n}.
\]
This should be understood in principal  valued sense, but from now on we  assume this without further mention. It is not difficult to see, and it actually follows from  Proposition \ref{change order}, that  if $u \in C^{\infty}(\Sigma)$  then  $\Delta^{\frac{s+1}{2}} u$ is a  smooth function on $\Sigma$. It is well known \cite{DPW, FFMMM} that by linearizing 
the fractional mean curvature at $\Sigma$ one obtains the following   Jacobi operator 
\beq \label{linear op}
L[u](x) : =   \Delta^{\frac{s+1}{2}} u(x)  + c_{s}^2(x)  \, u(x) ,
\eeq
where 
\[
c_{s}^2(x) = \int_{\Sigma}\frac{|\nu(y)- \nu(x)|^2 }{|y-x|^{n+1+s}}  \, d \Ha_y^{n}.
\]
We note that since $\Sigma$ is a smooth surface,  $c_{s}^2(\cdot)$ defines a smooth function on $\Sigma$. Again this is  not difficult to see and it follows 
from our analysis in Section 4. Moreover, since we assume $\Sigma$ is uniformly  $C^{1,1}$-regular, the $\alpha$-H\"older norm of $c_{s}^2$, for small $\alpha$, is uniformly  bounded (see Lemma \ref{lemma aux 2}).


As we mentioned in the introduction, the proof of the main theorem is based on regularity estimates for  nonlinear nonlocal parabolic equation. To this aim we need standard Schauder estimates for the fractional heat equation with a forcing term 
\beq  \label{parabolic eq}
\begin{cases}
&\partial_t u = \Delta^{\frac{s+1}{2}} u + f(x,t) +g(x) \qquad \text{on } \, \Sigma \times (0,T] \\
&u(x,0) = u_0(x)  \qquad \text{for  } \, x \in \Sigma.
\end{cases}
\eeq

 We prove the following Schauder estimate. We give the proof in the Appendix. 
\begin{theorem}
\label{parabolic est}
Assume that $f : \Sigma \times [0,T] \to \R$  and $u_0, g  :\Sigma \to \R$ are smooth and fix $\alpha \in (0,1-s)$. Then  \eqref{parabolic eq} has a unique smooth solution and it  holds 
\[
\sup_{0<t<T} \|u(\cdot,t)\|_{C^{1+s+\alpha}(\Sigma)} \leq C(1+T) \big( \|u_0\|_{C^{1+s+\alpha}(\Sigma)} +  \sup_{0<t<T}\|f(\cdot,t)\|_{C^{\alpha}(\Sigma)} + T \|g\|_{C^{1+s+\alpha}(\Sigma)} \big)
\]
and
\[
\sup_{0<t<T} \|u(\cdot,t)\|_{C^{0}(\Sigma)} \leq  \|u_0\|_{C^{0}(\Sigma)} + T \big(  \sup_{0<t<T}\|f(\cdot,t)\|_{C^{0}(\Sigma)} + \|g\|_{C^{0}(\Sigma)} \big)
\]
\end{theorem}

The second statement is in fact a simple consequence of the maximum principle. 


\section{Parametrization of the flow \eqref{flow}}

In this section we follow  \cite{HP} (see also \cite{MantegazzaBook}) and    parametrize the equation \eqref{flow} by using the height function over a smooth reference surface. Note first that since  $\pa E_0$ is a compact   $C^{1,1}$-hypersurface  we  find for any $\eps>0$ a smooth compact hypersurface $\Sigma$ such that we may write $\pa E_0$ as a graph over $\Sigma$,
\[
\pa E_0 = \{ x + h_0(x) \nu (x) :  x \in \Sigma \}
\]
with  $\|h_0\|_{C^{0}(\Sigma)} <  \eps$ and $\|h_0\|_{C^{2}(\Sigma)} \leq C$.  Indeed, we may first  fix a smooth surface $\Sigma_0$ such that 
\[
\pa E_0 = \{ x + \tilde{h}(x)\nu_{\Sigma_0}(x) : x \in \Sigma_{0}\},
\] 
where $\tilde h \in C^{1,1}(\Sigma_0) $ (note that $\|\tilde h\|_{C^0(\Sigma_0)}$ is not necessarily small).  By standard mollification argument we  find $\tilde h_\eps \in C^\infty(\Sigma_0)$ with 
\[
\|\tilde h_\eps - \tilde h \|_{C^0(\Sigma_0)}  \leq \eps \quad \text{and} \quad \|\tilde h_\eps \|_{C^2(\Sigma_0)} \leq C. 
\]
Thus we may  define $\Sigma = \{ x + \tilde{h}_\eps(x)\nu_{\Sigma_0}(x) : x \in \Sigma_{0}\}$.

From now on we assume that $\alpha$ is a positive number such that $\alpha < (1-s)/2$. We note that because  $\pa E_0$ is only $C^{1,1}$-regular  we have $\|\nu_\Sigma\|_{C^1(\Sigma)} \leq C$ but $\|\nu_\Sigma\|_{C^{1+s+\alpha}(\Sigma)} \leq C\eps^{-s - \alpha}$. This means that we have to be careful in our analysis whenever we have terms which depend on the norm $\|\nu_\Sigma\|_{C^{1+s+\alpha}(\Sigma)} $, because  we cannot bound it uniformly if we want $\Sigma$ to be close to $\pa E_0$. Note that 
\[
\|h_0\|_{C^\alpha(\Sigma)}  \leq  \|h_0\|_{C^0(\Sigma)}^{1-\alpha} \|h_0\|_{C^1(\Sigma)}^\alpha  \leq  \eps^{1-\alpha}
\]
and therefore even if $\|\nu_\Sigma\|_{C^{1+s+\alpha}(\Sigma)}$ is large   it still  holds
\beq \label{obvious}
\|h_0\|_{C^\alpha(\Sigma)}  \|\nu_\Sigma\|_{C^{1+s+\alpha}(\Sigma)} \leq C \eps^{1-s - 2\alpha} 
\eeq
 for $ \alpha < (1 -s)/2 $. Therefore for any $\delta>0$ we may choose $\eps$ small such that 
\beq
\label{product 3}
\|h_0\|_{C^{1+s+\alpha}(\Sigma)} +  \|h_0\|_{C^\alpha(\Sigma)}  \|\nu_\Sigma\|_{C^{1+s+\alpha}(\Sigma)}  \leq \delta.
\eeq
In particular, this implies $ \|h_0 \, \nu_\Sigma\|_{C^{1+s+\alpha}(\Sigma)}  \leq C \delta$. 

Our goal is to write the family of sets $(E_t)_{t \in (0,T]}$,  which is a solution of \eqref{flow},  as a graph over the reference surface $\Sigma$. To be more precise, we look for    a function $h \in C(\Sigma \times [0,T]) \cap C^\infty(\Sigma \times(0,T])$ such that the family of sets $E_t$ given by 
\[
\pa E_t = \{ x + h(x,t) \nu(x) :  x \in \Sigma \} \quad \text{and} \quad h(x,0) = h_0(x) 
\]
is a solution of \eqref{flow}. In this section we provide the  calculations which show that this leads to  the  equation  
\beq \label{intro flow}
\pa_t h =   L[h]   +  P(x, h, \nabla h) -  H_\Sigma^s(x),
\eeq
where $H_\Sigma^s$ is the fractional mean curvature of the reference surface $\Sigma$ and $L[\cdot]$ is the linear operator defined in  \eqref{linear op}. 
The precise formula for the remainder term $P$ is given in  Proposition \ref{final equation}.  Our goal  in the next section  is then   to show that for  $\delta>0$ small   the function $x \mapsto P(x, u, \nabla  u)$  satisfies  
\[
\|P(\cdot, u, \nabla  u) \|_{C^{\alpha}(\Sigma)}  \leq C \delta \|u \|_{C^{1+s+\alpha}(\Sigma)},
\]
  when $\| u \|_{C^{1+s+ \alpha}(\Sigma)} \leq \delta$. This means that we may treat \eqref{intro flow} as  a small perturbation of the fractional heat equation, i.e., \eqref{parabolic eq} with $f =  0$ and $g = 0$.

In order to   calculate \eqref{intro flow} we	  define the class of sets $\mathfrak{h}_\delta(\Sigma)$ such that  $E \in \mathfrak{h}_\delta(\Sigma)$ if its boundary can be written as 
\beq  \label{def h de}
\pa E = \{x + h_E(x)\nu(x) : x \in \Sigma \} \qquad \text{and} \quad \|h_E\|_{C^{1+s+\alpha}(\Sigma)} \leq \delta.
\eeq
In particular, if  $E \in \mathfrak{h}_\delta(\Sigma)$ then its  boundary   is a compact  $C^{1+s+\alpha}$-hypersurface.

We  begin with a standard  calculation. 
\begin{lemma}
\label{calculation}
Let $E \subset \R^{n+1}$ be a smooth bounded set, let $\Phi_{\tau}$ be a family of diffeomorphisms such that $\Phi_0(x) = x$, denote the velocity field by 
$X(x) = \frac{d}{d \tau} \big|_{\tau= 0} \Phi_{\tau}(x)$ and suppose $X \in C^{1+s+\alpha}(\Sigma)$. Then it holds
\[
-\frac{d}{d \tau} \big|_{\tau= 0} H_{\Phi_{\tau}(E)}^s (\Phi_{\tau}(x)) = 2 \int_{\pa E} \frac{(X(y) - X(x)) \cdot \nu_E(y)}{|y-x|^{n+1+s}} \, d \Ha_y^{n} .
\]
\end{lemma}

\begin{proof}
Let us denote $E_ {\tau} = \Phi_{\tau}(E)$ and $k_\e(z)=(|z|^{2}+\e)^{-\frac{n+s+1}{2}}$. It is enough to show that 
\[
\frac{d}{d \tau} \big|_{\tau= 0}  \int_{E_{\tau}} k_\e(\Phi_{\tau}(x)-y)dy  = \int_{\pa E} k_\e(x-y) (X(y) - X(x)) \cdot \nu_E(y)  \, d \Ha_y^{n}.
\]
Indeed, by repeating the same calculations for the second term in \eqref{def FMC} and letting $\eps \to 0$ yields  the result.  We split the above term  as 
\[
\frac{d}{d \tau} \big|_{\tau= 0}  \int_{E_{\tau}}  k_\e(\Phi_{\tau}(x)-y)dy  =  \underbrace{\frac{d}{d \tau} \big|_{\tau= 0}  \int_{E_{\tau}} k_\e(x - y) \, dy}_{= I}  +   \underbrace{\frac{d}{d \tau} \big|_{\tau= 0}  \int_{E} k_\e(\Phi_{\tau}(x) - y) \, dy}_{= II}. 
\]

Let us denote the Jacobian of $\Phi_{\tau}$ by $ J_{\Phi_{\tau}, \R^{n+1}}$. Since $ \frac{d}{d \tau} \big|_{\tau= 0} J_{\Phi_{\tau}, \R^{n+1}}=   \Div_{\R^{n+1}} X$, we may write the first term by change of variables as 
\[
\begin{split}
I &= \frac{d}{d \tau} \big|_{\tau= 0}  \int_{E} k_\e(\Phi_{\tau}(y)-x) \, J_{\Phi_{\tau}, \R^{n+1}}(y)\, dy \\
&=  \int_{E} \left(k_\e(y-x) \, (\Div_{\R^{n+1}}  X)(y) + D_y k_\e(y-x) \cdot X(y)  \right)\, dy  \\
&=  \int_{E} \Div_{\R^{n+1}} (k_\e(y-x) \, X(y) )\, dy\\
&= \int_{\pa E} k_\e(y-x) \, (X(y)\cdot \nu_E(y))\, d\Ha_y^{n}.
\end{split}
\]
By symmetry it holds $D_x k_\e(x -y) = -D_y k_\e(x -y)$. Therefore  we have for the second term 
\[
\begin{split}
II &=   \int_{E} (D_x k_\e(x - y) \cdot X(x) )  \, dy \\
&= - \int_{E}\Div_{\R^{n+1}} (k_\e(y-x) \, X(x) )\, dy\\
&= -\int_{\pa E} k_\e(y-x) \, (X(x)\cdot \nu_E(y))\,  d \Ha_y^{n}.
\end{split}
\] 

%
\end{proof}

We may  use Lemma \ref{calculation}  to write   the fractional  mean curvature  $H_E^s$ over the reference surface $\Sigma$. Let $E \in   \mathfrak{h}_\delta(\Sigma)$ with $\pa E = \{x + h(x)\nu(x) : x \in \Sigma \} $. 
We define  the sets $E_{t'}$ as $\pa E_{t'} :=  \{x +  t' h(x)\nu(x) : x \in \Sigma \}$, with $t' \in [0,1]$, and family of diffeomorphisms $\Phi_{t'h} : \Sigma \to \pa E_{t'}$ as 
\[
\Phi_{t'h}(x) = x +  t' h(x)\nu(x)  . 
\]
Then for $x \in \Sigma$ we have
\beq \label{integrate}
-H_E^s(x + h(x)\nu(x)) = -\int_0^1 \frac{d}{dt'}   H_{\Phi_{t'h}(\Sigma)}^s(\Phi_{t'h}(x)) \, dt' - H_\Sigma^s(x).
\eeq
We denote the tangential Jacobian of $\Phi_{t'h}$ by  $J_{\Phi_{t'h}}$ (see \cite{AFP} for details) and define $\Phi_\tau(x) := \Phi_{(t' +\tau)h} (\Phi_{t'h}^{-1}(x)) $. Note that  $\Phi_\tau : \pa E_{t'} \to \pa E_{t' +\tau} $ is a diffeomorphism and
\[
 \frac{d}{d \tau} \big|_{\tau= 0} \Phi_{\tau}(x) = h( \Phi_{t'h}^{-1}(x) ) \nu( \Phi_{t'h}^{-1}(x))  \qquad \text{for } \, x \in \pa E_{t'} .
\]
 We apply  Lemma \ref{calculation}   and change of variables to deduce
\[
\begin{split}
-\frac{d}{dt'}   &H_{\Phi_{t'h}(\Sigma)}^s(\Phi_{t'h}(x)) \\
 &= 2 \int_{\pa E_{t'}} \frac{1}{|y-x|^{n+1+s}} \left( h( \Phi_{t'h}^{-1}(y)) \nu( \Phi_{t'h}^{-1}(y))  -  h( \Phi_{t'h}^{-1}(x) ) \nu( \Phi_{t'h}^{-1}(x)) \right)  \cdot \nu_{E_{t'}}(y)\, d \Ha_y^{n} \\
&=  2 \int_{\Sigma}\frac{1}{|\Phi_{t'h}(x)-\Phi_{t'h}(y)|^{n+1+s}}  \left( h(y) \nu(y)  -  h(x) \nu(x) \right) \cdot \nu_{E_{t'}}(\Phi_{t'h}(y))  J_{\Phi_{t'h}}(y) \, d \Ha_y^{n}\\
&=  2 \int_{\Sigma}\frac{(h(y)- h(x))}{|\Phi_{t'h}(x)-\Phi_{t'h}(y)|^{n+1+s}} \,  (\nu(y)  \cdot \nu_{E_{t'}}(\Phi_{t'h}(y))  J_{\Phi_{t'h}}(y)) \, d \Ha_y^{n}\\
&\qquad +2\left(\int_{\Sigma}\frac{(\nu(y)- \nu(x)) \cdot   \nu_{E_{t'}}(\Phi_{t'h}(y)) }{|\Phi_{t'h}(x)-\Phi_{t'h}(y)|^{n+1+s}} \,  J_{\Phi_{t'h}}(y) \, d \Ha_y^{n}\right) \, h(x).
\end{split}
\]
  We may write the normal $ \nu_{E_{t'}}$ (see \cite[Section 1.5]{MantegazzaBook}) as
\begin{equation} \label{nu E_t}
\nu_{E_t'}(\Phi_{t'h}(y)) =  \frac{1}{J_{\Phi_{t'h}}} \big( (1  +  Q_1(y, t' h, t' \nabla h) ) \nu(y) +  Q_2(y, t' h, t' \nabla h) \big),
\end{equation}
where $Q_i$ are smooth functions which depend on the second fundamental form of $\Sigma$ and  $Q_i(y, 0 , 0) = 0$, for $i =1,2$ for all $y \in \Sigma$.  Moreover,  $Q_2$ takes values on the tangent space, i.e.,  $Q_2(y, \cdot, \cdot) \cdot \nu(y) = 0$.  We  may thus write  
\[
\begin{split}
2 \int_{\Sigma}&\frac{(h(y)- h(x))}{|\Phi_{t'h}(x)-\Phi_{t'h}(y)|^{n+1+s}} \,  (\nu(y)  \cdot \nu_{E_{t'}}(\Phi_{t'h}(y))  J_{\Phi_{t'h}}(y)) \, d \Ha_y^{n} \\
&=2 \int_{\Sigma}\frac{(h(y)- h(x))}{|\Phi_{t'h}(x)-\Phi_{t'h}(y)|^{n+1+s}} \, (1 + Q_1(y, t' h, t' \nabla h))  \, d \Ha_y^{n}\\
&= 2 \Delta^{\frac{s+1}{2}} h(x) \\
&\,\,\,\,\,\,\,+ 2 \int_{\Sigma} (h(y)- h(x))\left( \frac{1 + Q_1(y, t' h, t' \nabla h)}{|\Phi_{t'h}(x)-\Phi_{t'h}(y)|^{n+1+s}}- \frac{1}{|y-x|^{n+1+s}}\right)\, d \Ha_y^{n}.
\end{split}
\]
We write  $2 (\nu(y)- \nu(x)) \cdot \nu(y) = |\nu(y) -\nu(x)|^2$ and obtain by  \eqref{nu E_t}
\[
\begin{split}
2\int_{\Sigma}&\frac{(\nu(y)- \nu(x)) \cdot   \nu_{E_{t'}}(\Phi_{t'h}(y)) }{|\Phi_{t'h}(x)-\Phi_{t'h}(y)|^{n+1+s}} \,  J_{\Phi_{t'h}}(y) \, d \Ha_y^{n} \\
&=\int_{\Sigma}\frac{|\nu(y)- \nu(x)|^2}{|y-x|^{n+1+s}}  \, d \Ha_y^{n} \\
&\,\,\,\,\,\,\,+  \int_{\Sigma} |\nu(y)- \nu(x)|^2 \left(\frac{1 + Q_1(y, t' h, t' \nabla h)}{|\Phi_{t'h}(x)-\Phi_{t'h}(y)|^{n+1+s}} - \frac{1}{|y-x|^{n+1+s}}\right)\, d \Ha_y^{n} \\
&\,\,\,\,\,\,\,+ 2\int_{\Sigma}\frac{(\nu(y)- \nu(x)) \cdot   Q_2(y, t' h, t' \nabla h)}{|\Phi_{t'h}(x)-\Phi_{t'h}(y)|^{n+1+s}}\,  \, d \Ha_y^{n} .
\end{split}
\]

To shorten the notation we denote the kernel $K_u : \Sigma \times \Sigma \to [0,\infty]$ generated by $u \in C^{1+s+\alpha}(\Sigma)$  as
\beq \label{kernel 1}
K_u(y,x) := \frac{1}{| y -x + u(y) \nu(y)  - u(x)\nu(x)|^{n+1+s}} .
\eeq
Recall that $Q_1(y,0,0) = 0$. We may thus write 
\[
\frac{1 + Q_1(y, t' h, t' \nabla  h)}{|\Phi_{t'h}(x)-\Phi_{t'h}(y)|^{n+1+s}} - \frac{1}{|y-x|^{n+1+s}} = \int_0^{t'} \frac{d}{d \xi} \big( (1 + Q_1(y, \xi h, \xi \nabla  h))K_{\xi h}(y,x)\big)d \xi.
\]
We may finally   write the fractional mean curvature of $E \in \mathfrak{h}_\delta(\Sigma)$  by recalling the linear operator $L[\cdot]$ in  \eqref{linear op}, by \eqref{integrate} and by  the previous calculations 
\beq \label{fracMC lin}
-H_E^s(x + h(x)\nu(x))   = L[h](x)  - H_\Sigma^s(x) + R_{1,h}(x)+ R_{2,h}(x) \, h(x) .
\eeq
The remainder terms  $R_{1,u}$ and $R_{2,u}$ are defined for a generic function $u \in C^{1+s+\alpha}(\Sigma)$ with $\| u \|_{C^{1+s+\alpha}(\Sigma)} \leq \delta$ as
\beq \label{R_1}
R_{1,u}(x): =2 \int_0^{1} \int_0^{t'} \int_{\Sigma} (u(y)- u(x)) \frac{d}{d \xi} \big( (1 + Q_1(y, \xi u, \xi \nabla u))K_{\xi u}(y,x)\big)\, d \Ha_y^{n} d \xi dt'
\eeq
and 
\beq \label{R_2}
\begin{split}
R_{2,u}(x) :=  &\int_0^1 \int_0^{t'} \int_{\Sigma} |\nu(y)- \nu(x)|^2 \frac{d}{d \xi} \big( (1 + Q_1(y, \xi u, \xi \nabla u))K_{\xi u}(y,x) \big)\, d \Ha_y^{n} d \xi dt' \\
&+ 2 \int_0^1\int_{\Sigma}(\nu(y)- \nu(x)) \cdot   Q_2(y, t' u, t' \nabla u)\, K_{t'u} (y,x) \, d \Ha_y^{n} dt',
\end{split}
\eeq
where the kernel $K_u$  is  defined in \eqref{kernel 1}, and $Q_1, Q_2$ are smooth functions which satisfy $Q_1(y,0,0)= Q_2(y,0,0)  = 0$ for all $y \in \Sigma$. 


In order to write the flow \eqref{flow} as an equation we  recall from \cite{MantegazzaBook} that  the normal velocity of the flow $(E_t)_t$ given by  $E_t= \Phi_t(E)$, where $\Phi_t(x) = x + h(x,t)\nu(x) $ on $\Sigma$, is 
\[
V_t = \big( \nu_{E_t}(\Phi_{t}(x)) \cdot \nu(x)  \big) \partial_t h .
\] 
By choosing in $t'=1$  in \eqref{nu E_t} we have 
\[
\nu_{E_t}(\Phi_t(x)) \cdot \nu(x) = \frac{1}{J_{\Phi_{t}}} (1 + Q_1(y, h, \nabla h))  
\]
and the Jacobian can be written as $J_{\Phi_{t}} = 1 + Q_3(y, h, \nabla h)$. Here   $Q_1$ and $Q_3$ are smooth functions with  $Q_1(x,0,0) = Q_3(x,0,0)= 0$ for all $x \in \Sigma$.  Thus when $E_t \in \mathfrak{h}_\delta(\Sigma)$ for small enough $\delta$ we may write
\beq \label{normals}
V_t =\big(  \nu_{E_t}(\Phi_t(x)) \cdot \nu(x) \big) \partial_t h =  \frac{\partial_t h}{1 + Q(y, h(\cdot,t), \nabla h(\cdot, t))}  
\eeq
where $Q$ is a smooth function with $Q(x,0,0) =0$ for all $x \in \Sigma$.  We may finally write the equation for $h$ by combing \eqref{fracMC lin} and \eqref{normals}.  We state this in the following  proposition. 

\begin{proposition}
\label{final equation}
 Assume that the flow $(E_t)_{t \in (0,T]}$, with  $E_t \in \mathfrak{h}_\delta(\Sigma)$ for $t \in (0,T]$, is a classical solution of  \eqref{flow} starting from $E_0$ with $\pa E_0 =  \{ x + h_0(x)\nu(x) : x \in \Sigma\}$ and assume $\delta$ is small. Then the function $h \in C(\Sigma \times [0,T]) \cap C^\infty(\Sigma \times (0,T])$ with $\pa E_t = \{ x + h(x,t)\nu(x) : x \in \Sigma\}$ is a solution of the equation 
\beq \label{the equation}
\begin{split}
\partial_t h  &=  (1 + Q(x, h, \nabla h)) \Big( L[h]  - H_\Sigma^s(x) + R_{1,h(\cdot,t)}(x)+ R_{2,h(\cdot,t)}(x) \, h(x,t) \Big) \qquad \text{on } \, \Sigma \times (0,T]
\end{split}
\eeq
with $h(x,0) = h_0$. Here $L$ is the linear operator defined in  \eqref{linear op} and $H_\Sigma^s$ is the fractional mean curvature of the reference surface $\Sigma$. The remainder terms $R_{1,h(\cdot,t)}$ and $R_{2,h(\cdot,t)}$ are defined in 
\eqref{R_1} and \eqref{R_2} respectively and $Q$ is a smooth function with  $Q(x,0,0) = 0$  for all $x \in \Sigma$.

Conversely, if $h \in  C(\Sigma \times [0,T]) \cap C^\infty(\Sigma \times (0,T])$ is a  solution of \eqref{the equation} with 
\[
h(x,0) = h_0 \quad \text{and} \quad \sup_{0<t<T}\|h(\cdot,t)\|_{C^{1+s+\alpha}(\Sigma)} \leq \delta,
\]
 then $\pa E_t = \{ x + h(x,t)\nu(x) : x \in \Sigma\}$ defines a  family of sets which is a solution of \eqref{flow} starting from $E_0$.  
\end{proposition}

\section{Regularity estimates for the nonlocal operators}

In this section we study the spatial  regularity issues related to the equation \eqref{the equation} and, in particular,   the remainder  terms $R_{1,u}$ and $R_{2,u}$ defined  in \eqref{R_1} and \eqref{R_2}. As we mentioned in the previous section,  our goal is to prove that if  $\|u\|_{C^{1+s+\alpha}}$ is small  then  $R_{1,u}$ and $R_{2,u}$  are small in the $C^\alpha$-sense, which  then implies that we may regard the equation \eqref{the equation} as a linear equation with a small perturbation. We  study also the higher order regularity of   $R_{1,u}$ and $R_{2,u}$ in order to prove  that the solution of \eqref{the equation}  becomes  instantaneously  smooth.  The complicated structure of  $R_{1,u}$ and $R_{2,u}$  makes this section  challenging.

Throughout this section  $K$ denotes  a generic kernel, if not otherwise mentioned,  while $K_u$ is  the kernel defined in \eqref{kernel 1}. Next we define the class of  kernels which we will use throughout the section. 
\begin{definition} 
\label{S kappa}
Let $\kappa >0$ and $K : \Sigma \times \Sigma \to \R \cup \{\pm \infty \}$. We say that $K \in \mathcal{S}_\kappa$ if the following three conditions hold:
\begin{itemize} 
\item[(i)]   $K$ is continuous at every  $y,x \in \Sigma$, $x \neq y$, and  it holds 
\[
|K(y,x)| \leq \frac{\kappa}{| y -x|^{n+1+s}}.
\] 
\item[(ii)]  The function $x \mapsto K(y,x)$ is differentiable at every $x,y \in \Sigma$, $x \neq y$, and 
\[
| \nabla_x K(y,x)  | \leq    \frac{\kappa}{ |y-x|^{n+2+s}}.
\] 
\item[(iii)] The function 
\[
\psi(x) := \int_{\Sigma} (y-x)\, K(y,x) \, d\Ha_y^{n}
\]
is H\"older continuous with $ \|\psi\|_{C^{\alpha}(\Sigma)} \leq \kappa$. 
\end{itemize} 
\end{definition}

\begin{remark} 
 Throughout the paper we assume that $\Sigma$ is a compact hypersurface, but Definition \ref{S kappa} can be extended to  the case $\Sigma = \R^n$.  For instance   the autonomous kernel  $|y-x|^{-n-1-s}$ in $\R^n$ trivially satisfies the conditions (i)-(iii) for  $\kappa =  n+1+s$. 
\end{remark}

The first two conditions in Definition \ref{S kappa} state that the kernel $K$ behaves similarly as the model case $|y-x|^{-n-1-s}$, while the third condition is somewhat more involved. Indeed, it is not trivial to prove the condition (iii)  for $K_u$ defined in \eqref{kernel 1}, since $\Sigma$ is not flat and  thus there is no  cancellation due to symmetry  as in the case $\Sigma = \R^{n}$. (We will prove this in Lemma \ref{kernel lemma}.) However, it is important first to notice that there are cases when   we do not need the condition (iii) to prove H\"older continuity estimates. This is stated in the following useful auxiliary lemma.

\begin{lemma}
\label{lemma aux 2}
Assume that $K : \Sigma \times \Sigma \to \R \cup \{\pm \infty \}$ satisfies the conditions (i) and (ii) in Definition \ref{S kappa} with constant $\kappa>0$. Moreover, assume that  $F \in C(\Sigma \times \Sigma)$ satisfies the following:
\begin{itemize}
\item[(1)] For all $x,y \in \Sigma$ it holds
\[
|F(y,x)| \leq \kappa_0 |y-x|^{1+s+\alpha} .
\] 
\item[(2)] For all $x,y,z \in \Sigma$ with $|y-x| \geq 2 |z-x|$ it holds
\[
|F(y,z) - F(y,x)| \leq \kappa_0 |z-x|^{s+\alpha} \, |y-x| .
\]
\end{itemize}
Then the function
\[
\psi(x) = \int_\Sigma F(y,x) K(y,x) \, d \Ha_y^n
\]
is H\"older continuous and  $\|\psi\|_{C^{\alpha}(\Sigma)} \leq C \kappa_0 \kappa$.
\end{lemma}

\begin{proof} 
By the condition  (i) in Definition \ref{S kappa} and by the assumption (1) we immediately obtain $\|\psi\|_{C^{0}(\Sigma)} \leq C \kappa_0 \kappa$, because  the function $y \mapsto F(y, x) K(y, x)$  is  integrable over $\Sigma$ for every $x$. To show the H\"older continuity we may assume that $0 \in \Sigma$ and we need to show  $|\psi(z) - \psi(0)| \leq C \kappa_0 \kappa |z|^\alpha$ for $z \in \Sigma$ close to $0$. We divide the set $\Sigma$ into $\Sigma_- = \Sigma \cap \{|y| \leq 2 |z| \}$ and    $\Sigma_+ = \Sigma \cap \{|y| > 2 |z| \}$. For all $y \in \Sigma_-$  it holds by  the condition (i) in Definition \ref{S kappa} and by the assumption (1)  that 
\[
 \int_{\Sigma_-} |F(y,z)|| K(y,z)| \, d \Ha_y^n \leq  \kappa_0 \kappa \int_{\Sigma_-} \frac{1}{|y-z|^{n-\alpha}} \, d \Ha_y^n \leq C  \kappa_0 \kappa  \int_0^{3|z|} \rho^{\alpha-1}\, d \rho  \leq C  \kappa_0 \kappa |z|^{\alpha}. 
\]
Similarly it holds
\[
 \int_{\Sigma_-} |F(y,0)|| K(y,0)| \, d \Ha_y^n \leq C  \kappa_0 \kappa  |z|^{\alpha}. 
\]

On the other hand  it follows from the condition  (ii) in  Definition \ref{S kappa} that for all $y \in \Sigma_+$, i.e., $|y| \geq 2 |z|$, it holds
\[
|K(y,z) - K(y,0)| \leq C \kappa \frac{|z|}{|y|^{n+2+s}} \leq C \kappa \frac{|z|^{s+\alpha}}{|y|^{n+1+2s+\alpha}}. 
\]
This together with the condition (i) and with the assumptions (1) and (2) yield
\[
\begin{split}
 \int_{\Sigma_+} |F&(y,z) K(y,z)  - F(y,0) K(y,0)| \, d \Ha_y^n \\
&\leq  \int_{\Sigma_+} |F(y,z) - F(y,0)| |K(y,z)|  + |F(y,0)||K(y,z) - K(y,0)| \, d \Ha_y^n \\
&\leq C  \kappa_0 \kappa |z|^{s+\alpha} \int_{\Sigma_-} \frac{1}{|y|^{n+s}} \, d \Ha_y^n \\
&\leq  C  \kappa_0 \kappa |z|^{s+\alpha} \int_{2|z|}^{\infty} \rho^{-1-s}\, d \rho \leq C  \kappa_1 \kappa_2 |z|^{\alpha} .
\end{split}
\] 
 These imply $|\psi(z) -\psi(0)| \leq C  \kappa_1 \kappa_2 |z|^{\alpha}$. 
\end{proof}

We proceed by stating  first  the crucial regularity estimates we need  repeatedly  in the paper,   and  then prove that $K_u \in \mathcal{S}_\kappa$   (see Definition \ref{S kappa}) for bounded $\kappa$.  
\begin{lemma}
\label{lemma aux}
Let $K \in \mathcal{S}_\kappa$  (see Definition \ref{S kappa}) and assume $v_1 \in C^{1+s + \alpha}(\Sigma)$, $v_2 \in C^{s+\alpha}(\Sigma)$ and $v_3 \in C^{\alpha}(\Sigma)$.   Then the  function
\[
\psi(x) = \int_\Sigma (v_1(y) -v_1(x)) v_2(y)v_3(x) \, K(y,x)  \, d \Ha_y^{n}
\]
is H\"older continuous and
\[
\|\psi\|_{C^{\alpha}(\Sigma)} \leq  C\kappa\|v_1\|_{C^{1+s+\alpha}(\Sigma)}\|v_2\|_{C^{s+\alpha}(\Sigma)}\|v_3\|_{C^{\alpha}(\Sigma)}.
\]
\end{lemma}

\begin{proof} 
 We write  $\psi$ as
\[
\psi(x) =v_3(x)  \underbrace{\int_\Sigma (v_1(y) -v_1(x))(v_2(y) -v_2(x)) \, K(y,x)  \, d \Ha_y^{n}}_{=: \psi_1(x)} +  v_2(x)v_3(x) \underbrace{\int_\Sigma (v_1(y) -v_1(x)) \, K(y,x)  \, d \Ha_y^{n}}_{=:\psi_2(x)}.
\]
Note that  $\| \psi\|_{C^\alpha(\Sigma)} \leq \|v_3\|_{C^\alpha(\Sigma)}  \|\psi_1\|_{C^\alpha(\Sigma)}  + \|v_3\|_{C^\alpha(\Sigma)} \|v_2\|_{C^\alpha(\Sigma)} \|\psi_2\|_{C^\alpha(\Sigma)} $. Therefore it is enough to 
 estimate  $\|\psi_1\|_{C^\alpha(\Sigma)}$ and $\|\psi_2\|_{C^\alpha(\Sigma)}$. We define 
\[
F_1(y,x) := (v_1(y) -v_1(x))(v_2(y) -v_2(x)) .
\]
It is straightforward to check that  $F_1$ satisfies the assumptions of Lemma  \ref{lemma aux 2} with  $\kappa_0 \leq C  \, \|v_1\|_{C^{1}(\Sigma)}  \|v_2\|_{C^{s+\alpha}(\Sigma)}$.  
Therefore  Lemma \ref{lemma aux 2}  yields  $\|\psi_1\|_{C^{\alpha}(\Sigma)} \leq C \kappa \, \|v_1\|_{C^{1}(\Sigma)}  \|v_2\|_{C^{s+\alpha}(\Sigma)}$. We need thus to show that 
\beq 
\|\psi_2\|_{C^{\alpha}} \leq C\kappa  \|v_1\|_{C^{1+s+\alpha}(\Sigma)} .  
\eeq

To this aim we write $\psi_2$ as
\[
\psi_2(x) = \underbrace{\int_\Sigma \big(v_1(y) -v_1(x) - D_\tau v_1(x) \cdot (y -x)\big)\, K(y,x)  \, d \Ha_y^{n}}_{=: \psi_3(x)} + D_\tau v_1(x) \cdot \int_\Sigma  (y -x)\, K(y,x)  \, d \Ha_y^{n} .
\]
It follows immediately from the condition (iii) in Definition  \ref{S kappa} that the second term on the right-hand-side is H\"older-continuous with $C^\alpha$-norm bounded by $C \kappa \|v_1\|_{C^{1+\alpha}(\Sigma)}$. We need thus  to prove the  H\"older-continuity of $\psi_3$. We notice that  for every $x,y \in \Sigma$ it holds
\[
|v_1(y) -v_1(x) - D_\tau v_1(x) \cdot (y -x)| \leq  \|v_1\|_{C^{1+s+\alpha}}|y-x|^{1+s+\alpha}. 
\] 
Therefore  the function  
\beq
\label{aux F_3}
F_2(y,x) := v_1(y) -v_1(x) - D_\tau v_1(x) \cdot (y -x)  
\eeq
satisfies the assumption (1) of Lemma  \ref{lemma aux 2} with $\kappa_0 \leq    \|v_1\|_{C^{1+s+\alpha}(\Sigma)}$. Moreover for every $x,y,z \in \Sigma$ with $|y-x| \geq 2|x-z|$ it holds
\[
\begin{split}
\big| F_2(y,z) - F_2(y,x) \big|  &= \big| \big( v_1(y) - v_1(z)  -  D_\tau v_1(z) \cdot (y-z) \big)-  \big(v_1(y) - v_1(x) - D_\tau v_1(x) \cdot (y-x) \big) \big| \\
&= \big| \big( v_1(x) - v_1(z)  -  D_\tau v_1(z) \cdot (x-z) \big) + (D_\tau v_1(x) - D_\tau v_1(z)) \cdot (y-x)  \big| \\
 &\leq   \|v_1\|_{C^{1+s+\alpha}}  |z-x|^{1+s+\alpha} + \|v_1\|_{C^{1+s+\alpha}}  |z-x|^{s+\alpha} |y-x| \\
 &\leq  2 \|v_1\|_{C^{1+s+\alpha}}  |z-x|^{s+\alpha}|y-x|.
\end{split}
\]
Therefore  $F_2$ satisfies  the assumption (2) of Lemma  \ref{lemma aux 2} with $\kappa_0 \leq  2 \|v_1\|_{C^{1+s+\alpha}(\Sigma)}$, and we conclude by Lemma  \ref{lemma aux 2}  that 
\[
\|\psi_3\|_{C^{\alpha}} \leq C\kappa  \|v_1\|_{C^{1+s+\alpha}(\Sigma)} .
\]
\end{proof}

Let us now prove that the kernel $K_u$ defined in \eqref{kernel 1} belongs to the class $\mathcal{S}_\kappa$(see Definition \ref{S kappa}) for bounded $\kappa$, when    the norm $\|u \, \nu_\Sigma\|_{C^{1+s+\alpha}(\Sigma)}$ is small. Recall that this is a reasonable assumption by  \eqref{product 3}.  We denote
\[
\Phi_u(x) := x + u(x)\nu(x)
\]
and  write $K_u$ defined in \eqref{kernel 1}  as
\[
K_u(y,x) = \frac{1}{|\Phi_u(y) -\Phi_u(x)|^{n+1+s}}.
\] 
 We study also the linearization of   $K_u$, which means that for a given $w \in C^{1+s+\alpha}(\Sigma)$ we consider 
\beq 
\label{differentiate K}
\frac{d}{d\xi} \Big|_{\xi = 0} K_{u+ \xi w}(y,x)  = -\frac{n+1+s}{|\Phi_u(y) -\Phi_u(x)|^{n+3+s}} \big(  \Phi_u(y) -\Phi_u(x)\big) \cdot  \big(  w(y)\nu(y) -w(x) \nu(x)\big).
\eeq

\begin{lemma}
\label{kernel lemma}
Assume that  $u \in C^{1+s+\alpha}(\Sigma)$  is such that $\|u\|_{C^{1+s+\alpha}(\Sigma)} + \|u \,  \nu_\Sigma\|_{C^{1+s+\alpha}(\Sigma)} \leq \delta$  and $w \in C^{1+s+\alpha}(\Sigma)$. Then the following hold.
\begin{itemize}
\item[(a)] When $\delta$ is small enough the kernel $K_u$ defined in \eqref{kernel 1} belongs to the class $\mathcal{S}_{\kappa_1}$,  with $\kappa_1 \leq C$. 

\item[(b)]   When $\delta$ is small enough the kernel $\frac{d}{d\xi} \Big|_{\xi = 0} K_{u+ \xi w}$  belongs to the class $\mathcal{S}_{\kappa_2}$, with 
\[
\kappa_2 \leq C \|w \, \nu_\Sigma\|_{C^{1+s+\alpha}(\Sigma)} .
\]
\end{itemize}
\end{lemma}

\begin{proof} \textbf{Claim (a):}  We denote $\Phi_u(x) := x + u(x)\nu(x)$ and recall that 
\[
K_u(y,x) = \frac{1}{|\Phi_u(y) -\Phi_u(x)|^{n+1+s}}.
\]
It follows from the assumption $\|u \,  \nu_\Sigma\|_{C^{1+s+\alpha}(\Sigma)} \leq \delta$ that 
\beq 
\label{comp dist}
\frac{1}{2}|y-x| \leq |\Phi_u(y) -\Phi_u(x)| \leq 2 |y-x|,
\eeq
when $\delta$ is small. Therefore it is clear that $K_u$ satisfies the conditions (i) and (ii) in Definition \ref{S kappa} of $\mathcal{S}_{\kappa_1}$ with $\kappa_1 \leq C$.

The condition (iii) in Definition \ref{S kappa} is technically more involved to verify. We note that in principle we should regularize the kernel $K_u$  for the forthcoming calculations as in the proof of Lemma \ref{calculation}, but we ignore this since it can be done with obvious changes.   Recall that we need to show that the function
\beq \label{psi holder}
\psi(x) = \int_{\Sigma}(y-x)K_u(y,x)\, d \Ha_y^n
\eeq
is H\"older continuous. The idea is to use integration parts in order to write $\psi$ as a nonsingular integral.  To be more precise, we prove the following equality 
\beq 
\label{long}
\big(I + \tilde{Q}(x, D_\tau (u \, \nu_\Sigma)) \big) \overbrace{\int_{\Sigma}(y-x)K_u(y,x)\, d \Ha_y^n}^{=\psi(x)} =  \int_{\Sigma} F_{\Phi_u}(y,x) \, K_u(y,x) d \Ha_y^n,
\eeq
where $\tilde Q$ is a smooth  function with $\tilde Q(x, 0) = 0$ for all $x \in \Sigma$, and 
\beq 
\label{long 0}
\begin{split}
F_{\Phi_u}(y,x) =  &-\frac{H_\Sigma(y) \, \nu(y) }{(n-1+s)} |\Phi_u(y) - \Phi_u(x)|^{2} + \big((y-x)\cdot \nu(x)\big) \,  \nu(x)   \\
&-(D_\tau \Phi_u(y)-D_\tau \Phi_u(x))^T(\Phi_u(y) - \Phi_u(x)) \\
&- D_\tau \Phi_u(x)^T \big(\Phi_u(y) - \Phi_u(x) - D_\tau \Phi_u(x)(y-x)\big).
\end{split}
\eeq
Recall that we already know that $K_u$ satisfies the conditions (i) and (ii) in  Definition \ref{S kappa}.  The idea is then to show that  $F_{\Phi_u}$ defined in \eqref{long 0}  
satisfies the assumptions   of Lemma \ref{lemma aux 2}, which then implies that the RHS of \eqref{long} defines a H\"older continuous function. 

In order to show \eqref{long} we shorten the notation by $\Phi(x) = \Phi_u(x)$ and  notice that the tangential gradient of $y \mapsto |\Phi(y) - \Phi(x)|^{-n+1-s}$ is 
\[
\begin{split}
D_{\tau(y)} |\Phi(y) - \Phi(x)|^{-n+1-s} &=   -(n-1+s)\frac{D_\tau\Phi(y)^T (\Phi(y) - \Phi(x))}{|\Phi(y) - \Phi(x)|^{n+1+s}}  \\
&=   -(n-1+s)\frac{D_\tau\Phi(x)^T (\Phi(y) - \Phi(x))}{|\Phi(y) - \Phi(x)|^{n+1+s}}\\
&\,\,\,\,\,\,\,\,\,   -(n-1+s)\frac{(D_\tau\Phi(y) - D_\tau\Phi(x))^T(\Phi(y) - \Phi(x))}{|\Phi(y) - \Phi(x)|^{n+1+s}} .
\end{split}
\]
By the divergence theorem  it holds
\[
\int_{\Sigma} D_{\tau(y)} |\Phi(y) - \Phi(x)|^{-n+1-s}\, d \Ha_y^n = \int_{\Sigma} \frac{H_\Sigma(y) \, \nu(y)}{|\Phi(y) - \Phi(x)|^{n-1+s}}\, d \Ha_y^n . 
\]
Therefore the two previous equalities yield
\beq 
\label{long 1}
\begin{split}
D_\tau\Phi(x)^T  \int_{\Sigma} \frac{\Phi(y) - \Phi(x)}{|\Phi(y) - \Phi(x)|^{n+1+s}}\, d \Ha_y^n   = &-\frac{1}{(n-1+s)} \int_{\Sigma} \frac{H_\Sigma(y) \, \nu(y)}{|\Phi(y) - \Phi(x)|^{n-1+s}}\, d \Ha_y^n  \\
&- \int_{\Sigma}\frac{(D_\tau\Phi(y) - D_\tau\Phi(x))^T(\Phi(y) - \Phi(x))}{|\Phi(y) - \Phi(x)|^{n+1+s}}\, d \Ha_y^n .
\end{split}
\eeq
We write the term on the left-hand-side  as 
\beq 
\label{long 2}
\begin{split}
D_\tau\Phi(x)^T  \int_{\Sigma} \frac{\Phi(y) - \Phi(x)}{|\Phi(y) - \Phi(x)|^{n+1+s}}\, d \Ha_y^n   = &(D_\tau\Phi(x)^T \, D_\tau \Phi(x))   \int_{\Sigma} \frac{y-x}{|\Phi(y) - \Phi(x)|^{n+1+s}}\, d \Ha_y^n \\
&+D_\tau\Phi(x)^T  \int_{\Sigma} \frac{\Phi(y) - \Phi(x) - D_\tau \Phi(x) (y-x)}{|\Phi(y) - \Phi(x)|^{n+1+s}}\, d \Ha_y^n .
\end{split}
\eeq
The  equality  \eqref{long}   then follows from \eqref{long 1}, \eqref{long 2} and from the fact that 
\[
D_\tau\Phi(x)^T \, D_\tau \Phi(x) = I - \nu(x) \otimes \nu(x) + \tilde{Q}(x, D_\tau (u \, \nu_\Sigma)),
\]
where $\tilde{Q}(x,0) = 0$ for all $x \in \Sigma$.

When $\|u\, \nu_\Sigma\|_{C^{1}}\leq \delta$, with $\delta$ small enough, the matrix $I + \tilde{Q}(x, D_\tau (u \, \nu_\Sigma))$ is invertible. Therefore in order to show that 
\[
\psi(x) = \int_{\Sigma}(y-x)K_u(y,x)\, d \Ha_y^n
\]
 is H\"older continuous it is enough to show that the RHS in \eqref{long} is H\"older continuous. As we mentioned, we will do this by showing that $F_{\Phi_u}$ defined in \eqref{long 0} 
satisfies the assumptions  of Lemma \ref{lemma aux 2} with $\kappa_0 \leq C $.

First,  by using \eqref{comp dist} and $\|u \, \nu_\Sigma\|_{C^{1+s+\alpha}(\Sigma)} \leq \delta$  it is straightforward to check that
\[
F_1(y,x) = -\frac{1}{(n-1+s)} H_\Sigma(y) \, \nu(y) |\Phi_u(y) - \Phi_u(x)|^{2}   -(D_\tau \Phi_u(y)-D_\tau \Phi_u(x))^T(\Phi_u(y) - \Phi_u(x)) 
\]
satisfies the assumptions  of Lemma \ref{lemma aux 2} with $\kappa_0 \leq C$. Moreover we have  that   
\[
|\Phi_u(y) - \Phi_u(x) - D_\tau \Phi_u(x)(y-x) | \leq C  |y-x|^{1+s+\alpha}.
\]
Therefore, arguing as  with \eqref{aux F_3} we deduce that 
\[
F_2(y,x) = - D_\tau \Phi_u(x)^T \big(\Phi_u(y) - \Phi_u(x) - D_\tau \Phi_u(x)(y-x)\big)
\]
 also  satisfies the assumptions  of Lemma \ref{lemma aux 2} with $\kappa_0 \leq C $.

 We need yet to treat the term 
\[
F_3(y,x) =  \big((y-x)\cdot \nu(x)\big) \,  \nu(x) .
\]
Since $\Sigma$ is uniformly  $C^{1,1}$-regular hypersurface, there is a constant $C$ such that $|(y-x)\cdot \nu(x)| \leq C|y-x|^2$  for every $x,y \in \Sigma$. Therefore $F_3$ satisfies the assumption (1)  of Lemma \ref{lemma aux 2}. 
The assumption (2) is straightforward to verify but we do this for the reader's convenience.  We have  
\[
\begin{split}
|F_3(y,z) - F_3(y,x) |  &=  \big|  (\nu(z) \otimes \nu(z))  (y-z)  -   (\nu(x) \otimes \nu(x)) (y-x) \big| \\
&\leq \big|  (\nu(z) \otimes \nu(z) -  \nu(x) \otimes \nu(x)) (y-z)\big|  +  \big| (\nu(x) \otimes \nu(x)) ( (y-z) - (y-x)) \big| \\
&\leq \big|  (\nu(z) \otimes \nu(z) -  \nu(x) \otimes \nu(x))\big| |y-z|  +  \big| \nu(x) \cdot (x-z) \big| \\
&\leq C|z-x||y-z| + C|z-x|^2.
\end{split}
\]
Note that   $|y-x| \geq 2|z-x|$  implies $|y-z| \leq 2 |y-x|$. Therefore $F_3$ satisfies the assumption (2)  of Lemma \ref{lemma aux 2} and  the RHS of \eqref{long} is H\"older continuous. This proves the claim (a).

\medskip

\textbf{Claim (b):}

We denote $\pa_w K_{u} = \frac{d}{d\xi} \Big|_{\xi = 0} K_{u+ \xi w} $ for short. Note that from \eqref{differentiate K} and \eqref{comp dist} it follows
\[
|\pa_w K_{u}(y,x) | \leq \frac{C\|w\, \nu_\Sigma\|_{C^{1}(\Sigma)}}{|y-x|^{n+1+s}}
\]
for every $x,y \in \Sigma, x \neq y$. Therefore $\pa_w K_{u}$ satisfies the condition (i) in Definition \ref{S kappa} with $\kappa \leq C\|w\, \nu_\Sigma\|_{C^1(\Sigma)}$. It is straightforward to check that  
$\pa_w K_{u}$ satisfies also  the condition (ii) in Definition \ref{S kappa} with $\kappa \leq  C\|w\|_{C^1(\Sigma)}$. 

We need thus to verify the last condition, i.e., we show that 
\beq 
\label{kernel b}
\tilde{\psi}(x) = \int_{\Sigma}(y-x)\, \pa_w K_{u}  \, d \Ha_y^n
\eeq 
satisfies $\|\tilde{\psi}\|_{C^\alpha(\Sigma)} \leq C\|w \, \nu_\Sigma\|_{C^{1+s+\alpha}(\Sigma)}$. To this aim we recall that by \eqref{long}  for small $\xi$ it holds 
\beq 
\label{LONG}
\big(I + \tilde{Q}(x, D_\tau (u \, \nu_\Sigma+ \xi w\, \nu_\Sigma) ) \big) \int_{\Sigma}(y-x)K_{u+ \xi w}(y,x)\, d \Ha_y^n =  \int_{\Sigma} F_{\Phi_{u+ \xi w}}(y,x) \, K_{u+ \xi w}(y,x) d \Ha_y^n,
\eeq
where $F_{\Phi_{u+ \xi w}}$  is defined  in \eqref{long 0}.

Let us denote the RHS of \eqref{LONG} by 
\[
\varphi_\xi(x) := \int_{\Sigma} F_{\Phi_{u+ \xi w}}(y,x) \, K_{u+ \xi w}(y,x) d \Ha_y^n .
\]
By differentiating  we have
\[
\frac{\pa }{\pa \xi} \big|_{\xi = 0} \varphi_\xi(x) = \underbrace{\int_{\Sigma} \frac{\pa }{\pa \xi} \big|_{\xi = 0} F_{\Phi_{u+ \xi w}}(y,x) \, K_{u}(y,x) d \Ha_y^n}_{=: \varphi_{\xi,1}}  +  \underbrace{\int_{\Sigma}  F_{\Phi_{u}}(y,x) \, \pa_w K_{u}(y,x) d \Ha_y^n}_{=: \varphi_{\xi,2}}.
\]
Recall that  $\pa_w K_{u}$  satisfies  the conditions (i) and (ii) in Definition \ref{S kappa} with $\kappa \leq  C\|w \, \nu_\Sigma\|_{C^1(\Sigma)}$ and we already proved that $F_{\Phi_{u}}$ satisfies the assumptions (1) and (2) of  Lemma \ref{lemma aux 2}  with $\kappa_0 \leq C$. Lemma \ref{lemma aux 2} then implies that 
\[
\|\varphi_{\xi,2} \|_{C^\alpha(\Sigma)} \leq C\|w\, \nu_\Sigma\|_{C^1(\Sigma)} .
\]

To treat $\varphi_{\xi,1}$ we recall that we already proved that $K_{u} \in \mathcal{S}_\kappa$ for $\kappa \leq C$. We  simplify the expression  \eqref{long 0} by writing it as
\[
\begin{split}
F_{\Phi_{u+ \xi w}}(y,x)=  &-\frac{H_\Sigma(y) \, \nu(y) }{(n-1+s)} |\Phi_{u+ \xi w}(y) - \Phi_{u+ \xi w}(x)|^{2} + \big((y-x)\cdot \nu(x)\big) \,  \nu(x)   \\
&-D_\tau \Phi_{u+ \xi w}(y)^T(\Phi_{u+ \xi w}(y) - \Phi_{u+ \xi w}(x)) + D_\tau \Phi_{u+ \xi w}(x)^T  D_\tau \Phi_{u+ \xi w}(x) (y-x).
\end{split}
\]
We denote $\Phi_u'(x) = \frac{\pa }{\pa \xi} \big|_{\xi = 0} \Phi_{u + \xi w}(x) = w(x) \nu(x)$, differentiate the above equality and  obtain 
\[
\begin{split}
 \frac{\pa }{\pa \xi} \big|_{\xi = 0} F_{\Phi_{u+ \xi w}}(y,x)  =  &-\frac{2 H_\Sigma(y) \, \nu(y)}{(n-1+s)}  (\Phi_u(y) - \Phi_u(x))\cdot (\Phi_u'(y) -\Phi_u'(x) )   \\
&-D_\tau \Phi_u'(y)^T (\Phi_u(y) - \Phi_u(x)) -D_\tau \Phi_u(y)^T(\Phi_u'(y) - \Phi_u'(x)) \\
&+ \big( D_\tau \Phi_u'(x)^T  D_\tau \Phi_u(x)+  D_\tau  \Phi_u(x)^T D_\tau \Phi_u'(x)\big) (y-x) . 
\end{split}
\]
Since $\|\Phi_u' \|_{C^{1+s+\alpha}(\Sigma)}, \|D_\tau \Phi_u' \|_{C^{s+\alpha}(\Sigma)}  \leq C \| w \, \nu_\Sigma \|_{C^{1+s+\alpha}(\Sigma)} $ we  may use Lemma \ref{lemma aux}  to deduce  
\[
\|\varphi_{\xi,1} \|_{C^{1+s+\alpha}(\Sigma)}  \leq C  \| w \, \nu_\Sigma \|_{C^{1+s+\alpha}(\Sigma)}.
\]
The H\"older continuity of $\tilde{\psi}$, defined in \eqref{kernel b},  then follows from \eqref{LONG} and from the fact that 
\[
\|\frac{d}{d\xi} \Big|_{\xi = 0}\tilde{Q}(x, D_\tau (u \, \nu_\Sigma+ \xi w\, \nu_\Sigma) ) \|_{C^\alpha(\Sigma)} \leq \|w \, \nu_\Sigma\|_{C^{1+\alpha}(\Sigma)}.
\]
Hence we have 
\[
\|\tilde{\psi}\|_{C^\alpha(\Sigma)} \leq C\|w  \, \nu_\Sigma\|_{C^{1+s + \alpha}(\Sigma)}.
\]
This proves the claim (b). 
\end{proof}

\begin{remark} \label{flat S kappa 2}
It is clear that the results of Lemma \ref{lemma aux} and Lemma \ref{kernel lemma} hold also in  the case $\Sigma = \R^n$ if we assume that the functions  $v_1, v_2, v_3, u$ and $w$ are in the corresponding H\"older spaces globally, i.e., $v_i \in C^{1+s+\alpha}(\R^n)$ for $i=1,2,3$ and $\|u\|_{C^{1+s+\alpha}(\R^n)} \leq \delta$. 
\end{remark}

We may use the previous results to prove that when $\|u\|_{C^{1+s+\alpha}}$ is  small then  the remainder terms  $R_{1,u}$ and $R_{2,u}$ defined in \eqref{R_1} and \eqref{R_2} are small. This is stated more precisely in   the following proposition. 
Recall that we may ignore the dependence on $C^{1}$-norm of $\nu_\Sigma$, but we do however need to keep track on the dependence on higher norm of $\nu_\Sigma$ for later purpose. 
\begin{proposition}
\label{remainder estimate}
Assume $u \in C^{1 + s+\alpha}(\Sigma)$ is such that  $\|u\|_{C^{1 + s+\alpha}(\Sigma)} + \|u \, \nu_\Sigma\|_{C^{1 + s+\alpha}(\Sigma)} \leq \delta$, and let $R_{1,u}$ and $R_{2,u}$ be the functions defined in  \eqref{R_1} and in  \eqref{R_2} respectively. Then for $\delta>0$ small enough  it holds 
\[
\|R_{1,u}\|_{C^{\alpha}(\Sigma)} \leq C \delta  \, \|u\|_{C^{1+s+\alpha}(\Sigma)} \qquad \text{and} \qquad \|R_{2,u}\|_{C^{\alpha}(\Sigma)} \leq C \delta  \|\nu_\Sigma\|_{C^{1+s+\alpha}(\Sigma)}. 
\]
\end{proposition}

\begin{proof}
\textbf{Estimate for $R_{1,u}$:} \, Recall that 
\[
R_{1,u}(x) =2 \int_0^{1} \int_0^{t'} \int_{\Sigma} (u(y)- u(x)) \frac{d}{d \xi} \big( (1 + Q_1(y, \xi u, \xi \nabla u))K_{\xi u}(y,x)\big)\, d \Ha_y^{n} d \xi dt' .
\]
For later purpose we prove a slightly  more general claim. Assume that $u$ is as in the assumption, $v \in C^{1+s+\alpha}(\Sigma)$ and define
\beq 
\label{tarviit monesti}
\varphi(x) :=   \int_{\Sigma} (v(y)- v(x)) \frac{d}{d \xi} \big( (1 + Q_1(y, \xi u, \xi \nabla u))K_{\xi u}(y,x)\big)\, d \Ha_y^{n}.
\eeq
We claim that it holds 
\beq 
\label{tarviit}
\|\varphi\|_{C^{\alpha}(\Sigma)} \leq C \delta  \, \|v\|_{C^{1+s+\alpha}}
\eeq
for all $\xi \in [0,1]$. The estimate for $R_{1,u}$  then follows from \eqref{tarviit} by choosing $v = u$.

In order to prove \eqref{tarviit}  we write \eqref{tarviit monesti} as
\[
\begin{split}
\varphi(x) = &\int_{\Sigma}  (v(y)- v(x))  \, \big( \frac{d}{d \xi} Q_1(y, \xi u, \xi \nabla u)\big) \, K_{\xi u}(y,x)  \, d \Ha_y^{n} \\
&+\int_{\Sigma} (v(y)- v(x)) \,  (1 + Q_1(y, \xi u, \xi \nabla u)) \big(\frac{d}{d \xi} K_{\xi u}(y,x)\big) \, d \Ha_y^{n}.
\end{split}
\]
When $\delta$ is small  Lemma \ref{kernel lemma} yields $K_{\xi u} \in \mathcal{S}_{\kappa_1}$ with $\kappa_1 \leq C$ and $ \frac{d}{d \xi} K_{\xi u} \in \mathcal{S}_{\kappa_2}$ with $\kappa_2 \leq C \|u\, \nu_\Sigma\|_{C^{1+s+\alpha}(\Sigma)} \leq C \delta$ for all $\xi \in [0,1]$. We note that the functions $Q_1$ and $Q_2$ in \eqref{nu E_t} depend on the second fundamental form  of $\Sigma$ such that $\|Q_i(x, u , \nabla u)\|_{C^{\beta}(\Sigma)} \leq C(\|u \|_{C^{1+\beta}(\Sigma)} + \|u \nu_\Sigma \|_{C^{1+\beta}(\Sigma)}  )$ for $\beta \in (0,1)$. Therefore it holds
\beq
\label{R_1 extra 2}
 \| 1 + Q_1(y, \xi u, \xi \nabla u) \|_{C^{s+\alpha}(\Sigma)} \leq C
\eeq
and we also have 
\beq
\label{R_1 extra 1}
 \| \frac{d}{d \xi} Q_1(y, \xi u, \xi \nabla u) \|_{C^{s+\alpha}(\Sigma)} \leq  C(\|u \|_{C^{1+s+\alpha}(\Sigma)} + \|u \nu_\Sigma \|_{C^{1+s+\alpha}(\Sigma)}  ) \leq C \delta 
\eeq
for all $\xi \in [0,1]$. Hence, the estimate \eqref{tarviit} follows  from Lemma \ref{lemma aux}.

\medskip

\textbf{Estimate for $R_{2,u}$:} \, Recall that 
\[
\begin{split}
R_{2,u}(x) :=  &\int_0^1 \int_0^{t'} \int_{\Sigma} |\nu(y)- \nu(x)|^2 \frac{d}{d \xi} \big( (1 + Q_1(y, \xi u, \xi \nabla u))K_{\xi u}(y,x) \big)\, d \Ha_y^{n} d \xi dt' \\
&+ 2 \int_0^1\int_{\Sigma}(\nu(y)- \nu(x)) \cdot   Q_2(y, t' u, t' \nabla u)\, K_{t'u} (y,x) \, d \Ha_y^{n} dt'.
\end{split}
\]
We use  $|\nu(y)- \nu(x)|^2  =  -2 \nu(x) \cdot   (\nu(y)- \nu(x))$  and   write the first term as  
\[
\begin{split}
\psi_1(x) &= \int_{\Sigma} |\nu(y)- \nu(x)|^2 \frac{d}{d \xi} \big( (1 + Q_1(y, \xi u, \xi \nabla u))K_{\xi u}(y,x) \big)\, d \Ha_y^{n} \\
&= -2 \nu(x) \cdot \int_{\Sigma} (\nu(y)- \nu(x)) \frac{d}{d \xi} \big( (1 + Q_1(y, \xi u, \xi \nabla u))K_{\xi u}(y,x) \big)\, d \Ha_y^{n}
\end{split}
\]
The function inside the integral is of type \eqref{tarviit monesti} with $v = \nu$ and therefore \eqref{tarviit} implies    $\|\psi_1\|_{C^\alpha(\Sigma)}\leq C \delta  \|\nu_\Sigma\|_{C^{1+s+\alpha}(\Sigma)}$.

Let us fix $t' \in [0,1]$. We need yet to prove that the function 
\[
\psi_2(x) =  \int_{\Sigma}(\nu(y)- \nu(x)) \cdot   Q_2(y, t' u, t' \nabla u)\, K_{t'u} (y,x) \, d \Ha_y^{n}
\]
satisfies $\|\psi_2\|_{C^\alpha(\Sigma)}\leq C\delta$. To this aim we recall that  we may estimate 
\[
\|Q_2(x, t' u, t' \nabla u) \|_{C^{\alpha+s}(\Sigma)} \leq C(\|u \|_{C^{1+s+\alpha}(\Sigma)} + \|u \nu_\Sigma \|_{C^{1+s+\alpha}(\Sigma)}  ) \leq C \delta.
\]
Since $K_{t'u} \in \mathcal{S}_{\kappa_1}$ with $\kappa_1 \leq C$,  Lemma \ref{lemma aux}  implies  $\|\psi_2\|_{C^\alpha(\Sigma)}\leq C\delta  \|\nu_\Sigma\|_{C^{1+s+\alpha}(\Sigma)}$. 
\end{proof}

We need similar estimate  as Proposition \ref{remainder estimate} for the linearization of  the remainder terms  $R_{1,u}$ and  $R_{2,u}$. 

\begin{proposition}
\label{remainder linear}
Assume $u, w \in C^{1 +s+\alpha}(\Sigma)$ and   $\|u\|_{C^{1 + s+\alpha}(\Sigma)}  + \|u \, \nu_\Sigma\|_{C^{1 + s+\alpha}(\Sigma)}  \leq \delta$. Let $R_{1,u + \xi w}$ and $R_{2,u + \xi w}$ be functions defined in  \eqref{R_1} and in  \eqref{R_2} respectively and define
\[
\pa_w R_{1,u}(x) = \frac{d}{d \eta} \Big|_{\eta = 0}  R_{1,u + \eta w}(x) \qquad \text{and} \qquad \pa_w R_{2,u}(x) = \frac{d}{d \eta} \Big|_{\eta = 0}  R_{2,u + \eta w}(x) .
\]
Then for $\delta$ small    it holds
\[
\|\pa_w R_{1,u}\|_{C^{\alpha}(\Sigma)} \leq  C\delta \|w\|_{C^{1 +s+\alpha}(\Sigma)} + C_{\Sigma}\|w\|_{C^{0}(\Sigma)}
\]
and
\[
 \|\pa_w R_{2,u} \|_{C^{\alpha}(\Sigma)} \leq C  \|\nu_\Sigma\|_{C^{1+s+\alpha}(\Sigma)} \|w\|_{C^{1 +s+\alpha}(\Sigma)} +  C_{\Sigma}\|w\|_{C^{0}(\Sigma)}. 
\]
Here $C$ is a uniform constant while $C_{\Sigma}$ depends on $\|\nu_\Sigma \|_{C^{1 +s+\alpha}(\Sigma)}$.
\end{proposition}

\begin{proof}
This time we prove the claim only  for $\pa_w R_{1,u}$ since the argument for $\pa_w R_{2,u} $ is similar. We differentiate   $R_{1,u + \eta w}$, defined in \eqref{R_1},   and obtain 
\beq \label{linear long}
\begin{split}
\pa_w &R_{1,u}(x) = \frac{d}{d \eta} \Big|_{\eta = 0}  R_{1,u + \eta w}(x)\\
&=2 \int_0^{1} \int_0^{t'} \int_{\Sigma} (w(y)- w(x)) \frac{d}{d \xi} \big( (1 + Q_1(y, \xi u, \xi \nabla u))K_{\xi u }(y,x)\big)\, d \Ha_y^{n} d \xi dt'\\
&+2  \int_0^{1}  \int_{\Sigma} (u(y)- u(x)) \frac{d}{d \eta} \Big|_{\eta = 0} Q_1(y, t' ( u + \eta w), t' \nabla( u + \eta w))  \, K_{t' u }(y,x) \, d \Ha_y^{n} dt' \\ 
&+ 2  \int_0^{1}  \int_{\Sigma} (u(y)- u(x)) ( 1+ Q_1(y, t' u, t' \nabla u)  ) \left(   \frac{d}{d \eta} \Big|_{\eta = 0} K_{t' u + t'\eta w }(y,x) \right)\, d \Ha_y^{n}  dt'
\end{split}
\eeq
Note that the first term is of type \eqref{tarviit monesti} with $v = w$  and therefore \eqref{tarviit} implies that it is H\"older continuous and its $C^\alpha$-norm is bounded by $C \delta \|w\|_{C^{1+s+\alpha}(\Sigma)}$. Concerning the two last terms in \eqref{linear long}, note first  that   Lemma \ref{kernel lemma} (b) yields  $\frac{d}{d \eta} \Big|_{\eta = 0} K_{t' u + t'\eta w } \in \mathcal{S}_{\kappa_2}$ with $\kappa_2 \leq C \| w \, \nu_\Sigma\|_{C^{1+s+\alpha}(\Sigma)}$. By the interpolation inequality in Lemma \ref{aubinlemma} we may estimate
\[
 \| w \nu_\Sigma\|_{C^{1+s+\alpha}(\Sigma)} \leq C  \|w\|_{C^{1 +s+\alpha}(\Sigma)} +  C_{\Sigma}\|w\|_{C^{0}(\Sigma)}.
\]
 Moreover, we have by \eqref{R_1 extra 2}  
\[
\| 1+ Q_1(y, t' u, t'\nabla u) \|_{C^{s+\alpha}(\Sigma)} \leq C 
\]
and as in  \eqref{R_1 extra 1}  we have
\[
\begin{split}
\| \frac{d}{d \eta} \Big|_{\eta = 0} Q_1(y, t' ( u + \eta w), t' \nabla( u + \eta w)) \|_{C^{s+\alpha}(\Sigma)} &\leq C( \|w \|_{C^{1+s +\alpha}(\Sigma)} + \|w  \nu_\Sigma \|_{C^{1+s +\alpha}(\Sigma)})\\
&\leq C  \|w\|_{C^{1 +s+\alpha}(\Sigma)} +  C_{\Sigma}\|w\|_{C^{0}(\Sigma)}.
\end{split}
\]
Thus we deduce by Lemma \ref{lemma aux} that the two last terms in \eqref{linear long} are H\"older continuous  with $C^\alpha$-norms  bounded by $C \delta \|w\|_{C^{1+s+\alpha}(\Sigma)}  + C_{\Sigma}\|w\|_{C^{0}(\Sigma)}$. Hence, we have 
\[
\|\pa_w R_{1,u}\|_{C^{\alpha}(\Sigma)} \leq  C\delta \|w\|_{C^{1+s+\alpha}(\Sigma)}  + C_{\Sigma}\|w\|_{C^{0}(\Sigma)} .
\]
\end{proof}

At the end of the section we study how to  control the higher order norms of  $R_{1,u}$ and $R_{2,u}$    in order to differentiate  the equation \eqref{the equation}  with respect to $x$. Moreover, even if the fractional Laplacian 
is linear it is not obvious how to bound its higher order covariant derivatives.  Before that we remark on  how we write the derivative  of the function 
\[
\psi(x) = \int_{\Sigma} G(y,x) \, d \Ha_y^n
\] 
with respect to  a  vector field $X \in \mathscr{T}(\Sigma)$. First it holds
\[
 \nabla_X \psi(x) = \int_{\Sigma}  \nabla_{X(x)}   G(y,x) \, d \Ha_y^n,
\]
where $\nabla_{X(x)}  G(y,x)$ denotes the derivative of $G(y,\cdot)$ with respect to $X$.  On the other hand it holds
\[
\Div_{y} ( G(y,x) X(y)) =  \nabla_{X(y)}  G(y,x) + G(y,x)  \Div(X)(y)  .
\]
Therefore we have by the divergence theorem that
\beq
\label{derivointi kaava}
\nabla_X \psi(x) = \int_{\Sigma} (\nabla_{X(y)} +  \nabla_{X(x)})   G(y,x)\, d \Ha_y^n + \int_{\Sigma} G(y,x) \, \text{div}(X)(y) \, d \Ha_y^n ,
\eeq
where
\[
(\nabla_{X(y)} +  \nabla_{X(x)})   G(y,x) = \nabla_{X(y)} G(y,x) +  \nabla_{X(x)}  G(y,x) .
\]

The following proposition  gives us a formula to commute differentiation and the fractional Laplacian. In the following, and in the rest of the paper,  $C_k$ denotes a constant which depends on $k$ and 
on $\Sigma$ in an unspecified way while $C$ denotes a  uniform constant.  
\begin{proposition}
\label{change order}
Let $X_1, \dots, X_k \in \mathscr{T}(\Sigma)$ be  vector fields  such that  $\|X_i\|_{C^{k+2}(\Sigma)} \leq 1$ for $i =1, \dots, k$  and assume $u \in C^\infty(\Sigma)$. Then 
\[
\nabla_{X_k} \cdots \nabla_{X_1}  \big( \Delta^{\frac{1+s}{2}} u \big) = \Delta^{\frac{1+s}{2}}\big( \nabla_{X_k} \cdots \nabla_{X_1} u \big) + \pa^{k+s} u,  
\] 
where $\pa^{k+s} u$ denotes a function which satisfies $\|\pa^{k+s} u\|_{C^\alpha(\Sigma)} \leq C_k \|u\|_{C^{k+s+\alpha}(\Sigma)}$. Moreover,   it holds 
\[
\| \Delta^{\frac{1+s}{2}} u \|_{C^{k+\alpha}(\Sigma)} \leq  C_k \|  u \|_{C^{k+1+s+\alpha}(\Sigma)} 
\]
for every $k \in \N$.
\end{proposition}

\begin{proof}
Let us denote $K(y,x) = |y-x|^{-n-1-s}$ and let $X_1, \dots, X_k \in \mathscr{T}(\Sigma)$ be as in the assumption. We define 
$\pa_{1} K(y,x) := \nabla_{X_1(y)} K(y,x) +\nabla_{X_1(x)}K(y,x)$ and  $\pa_{j} K(y,x)$,  for $2 \leq j \leq k$,  recursively  as
\[
\pa_{j} K(y,x) := \nabla_{X_j(y)} \pa_{j-1} K(y,x) +  \nabla_{X_j(x)}\pa_{j-1} K(y,x) .
\]
 We begin by claiming that $\pa_{k} K(y,x)$ satisfies the conditions (i)-(iii) in Definition \ref{S kappa} with $\kappa \leq C_k$, i.e., $\pa_{k} K(y,x) \in \mathcal{S}_{C_k}$. Note that the constant does not 
depend on the chosen vector fields $X_1, \dots, X_k$ once they  satisfy $\|X_i\|_{C^{k+2}(\Sigma)} \leq 1$. 

It is straightforward to check that $\pa_{k} K(y,x)$ satisfies the conditions (i) and (ii) in the Definition \ref{S kappa} with $\kappa \leq C_k$ for some $C_k$. We need thus to prove the 
condition (iii).  We prove this by induction  and  fix $X  \in \mathscr{T}(\Sigma)$ such that  $\|X\|_{C^{3}(\Sigma)} \leq 1$. We need first to show that the function 
\beq \label{psi vaihto}
\psi_1(x) = \int_{\Sigma} (y-x) \pa_{1} K(y,x) \, d \Ha_y^n 
\eeq
is $\alpha$-H\"older continuous.

The formula \eqref{long} in the case $u = 0$ reads as 
\beq \label{long 22}
\int_{\Sigma} (y-x) K(y,x) \, d \Ha_y^n  = \int_{\Sigma} F(y,x) K(y,x) \, d \Ha_y^n,  
\eeq
where 
\[
F(y,x) = -\frac{H_\Sigma(y) \, \nu(y) }{(n-1+s)} |y-x|^{2} + \big((y-x)\cdot \nu(y)\big) \,  \nu(y). 
\]
We differentiate \eqref{long 22} with respect to $X(x)$ and obtain  by \eqref{derivointi kaava}
\[
\int_{\Sigma} (y-x) \pa_{1} K(y,x) \, d \Ha_y^n  = \underbrace{\int_{\Sigma} F(y,x)\pa_{1} K(y,x) \, d \Ha_y^n}_{=: \tilde{\psi}_1} +   \underbrace{\int_{\Sigma} F_1(y,x) K(y,x) \, d \Ha_y^n}_{=: \varphi_1},
\]
where 
\[
F_1(y,x) =  (\nabla_{X(y)} +  \nabla_{X(x)}) (F(y,x) - (y-x)) + \Div X(y) (F(y,x) - (y-x)). 
\]
First, recall that in the proof of Lemma \ref{kernel lemma} we already verified that $F$ satisfies the  assumptions  of Lemma \ref{lemma aux 2}. Since $\pa_{1} K(y,x) $ satisfies the conditions (i) and (ii) in Definition \ref{S kappa} with  
$\kappa \leq C_1$,  Lemma \ref{lemma aux 2} yields  $\|\tilde{\psi}_1\|_{C^\alpha(\Sigma)} \leq C_1$. Second, we may write $F_1$ as
\[
F_1(y,x) =  \sum_{j=1}^{N_1}  (v_{1,i}(y) - v_{1,i}(x)) v_{2,i}(y) v_{3,i}(x),
\]
where $v_{j,i}$ are such that $\|v_{j,i}\|_{C^2(\Sigma)} \leq C_1$.  Moreover, by Lemma \ref{kernel lemma} it holds $K(y,x) \in \mathcal{S}_\kappa$ with $\kappa \leq C$ and    we may thus use Lemma \ref{lemma aux} to conclude that   $\|\varphi_1\|_{C^\alpha(\Sigma)} \leq C_1$. Hence 
\[
\|\psi\|_{C^\alpha(\Sigma)} \leq C_1
\]
and therefore $\pa_{1} K(y,x) \in \mathcal{S}_\kappa$ with $\kappa \leq C_1$.

We argue by induction  and assume that $\pa_{k-1} K(y,x) \in \mathcal{S}_\kappa$ with $\kappa \leq C_{k-1}$.  Note that this holds for any vector fields $X_1, \dots, X_{k-1}$ with $\|X_i\|_{C^{k+1}(\Sigma)}\leq 1$. Let us  fix  $X_1, \dots, X_k$ as in the assumption.  We differentiate \eqref{long 22} with respect to $X_1, \dots, X_k$ and  obtain by \eqref{derivointi kaava}
\[
\int_{\Sigma} (y-x) \pa_{k} K(y,x) \, d \Ha_y^n  = \underbrace{\int_{\Sigma} F(y,x)\pa_{k} K(y,x) \, d \Ha_y^n}_{=: \tilde{\psi}_k} + \underbrace{\sum_{i = 0}^{k-1}  \int_{\Sigma} \tilde{F}_i(y,x)\pa_{i} K(y,x) \, d \Ha_y^n}_{=:\varphi_k}.
\]
Here $\tilde F_i$ can be written as
\[
\tilde F_i(y,x) = \sum_{j=1}^{N_k}  (v_{1,i}(y) - v_{1,i}(x)) v_{2,i}(y) v_{3,i}(x),
\]
where $v_{j,i}$ are such that $\|v_{j,i}\|_{C^2(\Sigma)} \leq C_k$.  Again we recall  that  $F$ satisfies the  assumptions  of Lemma \ref{lemma aux 2} and  $\pa_{k} K(y,x) $ satisfies the conditions (i) and (ii) in Definition \ref{S kappa}  for $\kappa \leq C_{k}$.  Lemma \ref{lemma aux 2}  then yields  $\|\tilde{\psi}_k\|_{C^\alpha(\Sigma)} \leq C_k$. To prove the H\"older continuity of $\varphi_k$ we recall that by induction assumption  $\pa_{j} K(y,x) \in \mathcal{S}_\kappa$ with $\kappa \leq C_{k-1}$ for 
every $j \leq k-1$. We may thus use Lemma \ref{lemma aux} to deduce $\|\varphi_k\|_{C^\alpha(\Sigma)} \leq C_{k}$. Therefore we conclude  that 
\[
\psi_k(x) =  \int_{\Sigma} (y-x) \pa_{k} K(y,x) \, d \Ha_y^n 
\] 
is H\"older continuous with $\|\psi_k\|_{C^\alpha(\Sigma)} \leq C_k$ and thus $\pa_{k} K(y,x)$ satisfies the condition (iii) in Definition \ref{S kappa} with $\kappa \leq C_k$.

We prove the claim by first choosing $\tilde X_1, \dots, \tilde X_l$, with $1 \leq l \leq k$, such that  $\|\tilde X_i\|_{C^{k+2}(\Sigma)} \leq 1$ for $1 \leq i \leq l$.  Recall  that the function $\tilde X_l \cdots \tilde X_1   u$ is defined 
recursively as $\tilde X_1   u = \nabla_{\tilde X_1} u $ and   $\tilde X_{j+1} \tilde X_j \cdots \tilde X_1   u =  \nabla_{\tilde X_{j+1}} (\tilde X_j \cdots \tilde X_1   u ) $ for $j \geq 1$. We apply \eqref{derivointi kaava} $l$ times  for $\tilde X_1, \dots, \tilde X_l$ and obtain 
\[
\begin{split}
\tilde X_l \cdots \tilde X_1 &\Delta^{\frac{1+s}{2}} u(x)  = 2\int_{\Sigma} (\tilde X_l \cdots \tilde X_1   u(y) - \tilde X_l \cdots \tilde X_1  u(x)) K(y,x)  \, d \Ha_y^n  \\
&+2\int_{\Sigma} (u(y) -u(x)) \pa_l K(y,x)  \, d \Ha_y^n + \sum_{j=0}^{l-1} \int_{\Sigma} (\pa^{l-1-j} u(y) -\pa^{l-1-j} u(x)) v_j(y) \pa_j K(y,x)  \, d \Ha_y^n,
\end{split}
\]
where $\pa^{l-1-j} u$ denotes a function which  satisfies $\|\pa^{l-1-j} u\|_{C^{1+s+\alpha}(\Sigma)} \leq C_{l-1} \| u\|_{C^{l-j+s+\alpha}(\Sigma)} $ and $v_j$ are  such that $\|v_j\|_{C^2(\Sigma)} \leq C_l$. By using  Lemma \ref{lemma aux} and    $\pa_{j} K(y,x) \in \mathcal{S}_{\kappa}$ with $\kappa \leq C_j$  we deduce  
\beq \label{tilde}
\tilde X_l \cdots \tilde X_1 (\Delta^{\frac{1+s}{2}} u) = \Delta^{\frac{1+s}{2}} (\tilde X_l \cdots \tilde X_1 u) + \pa^{l+s} u , 
\eeq
where $ \pa^{l+s} u$ denotes a function which satisfies $\|\pa^{l+s} u\|_{C^\alpha(\Sigma)} \leq C_l \|u\|_{C^{l+s+\alpha}(\Sigma)}$. Hence, we deduce by   \eqref{normi kikka}, Lemma  \ref{lemma aux}  and   \eqref{tilde}  that
\[
\| \Delta^{\frac{1+s}{2}} u \|_{C^{l+\alpha}(\Sigma)} \leq C_l \big(\|  u \|_{C^{l+1+s+\alpha}(\Sigma)} + \| \Delta^{\frac{1+s}{2}} u \|_{C^{l-1+\alpha}(\Sigma)}\big)
\]
for every $l \leq k$.  Note that Lemma \ref{lemma aux} implies $\| \Delta^{\frac{1+s}{2}} u\|_{C^\alpha(\Sigma)} \leq C \|u\|_{C^{1+s+\alpha}(\Sigma)} $. Therefore 
by iterating the above inequality  for $l = 1, \dots, k$ we obtain
\beq \label{tilde 2}
\| \Delta^{\frac{1+s}{2}} u \|_{C^{k+\alpha}(\Sigma)} \leq  C_k \|  u \|_{C^{k+1+s+\alpha}(\Sigma)} .
\eeq
This implies  the second statement. 

Let $X_1, \dots, X_k$ be as in the assumption. We deduce from  \eqref{tilde} that 
\[
 X_k \cdots  X_1 (\Delta^{\frac{1+s}{2}} u) = \Delta^{\frac{1+s}{2}} (X_k \cdots  X_1 u) + \pa^{k+s} u,
\] 
 where  $ \pa^{k+s} u$ denotes a function which satisfies $\|\pa^{k+s} u\|_{C^\alpha(\Sigma)} \leq C_k \|u\|_{C^{k+s+\alpha}(\Sigma)}$. By \eqref{derivointi kikka} 
we have
\[
 \nabla_{X_k} \cdots  \nabla_{X_1} (\Delta^{\frac{1+s}{2}} u) = \Delta^{\frac{1+s}{2}} ( \nabla_{X_k} \cdots  \nabla_{X_1} u)  + \pa^{k-1}(\Delta^{\frac{1+s}{2}} u) + \pa^{k+s} u,
\]
where $\pa^{k-1}(\Delta^{\frac{1+s}{2}} u) $ denotes a function which satisfies 
\[
\| \pa^{k-1}(\Delta^{\frac{1+s}{2}} u)\|_{C^\alpha(\Sigma)} \leq C_k \| \Delta^{\frac{1+s}{2}} u\|_{C^{k-1+\alpha}(\Sigma)}.
\] 
The estimate \eqref{tilde 2} applied to  $(k-1)$ yields
\[
\| \pa^{k-1}(\Delta^{\frac{1+s}{2}} u)\|_{C^\alpha(\Sigma)} \leq C_k \|  u \|_{C^{k+s+\alpha}(\Sigma)} 
\]
and the claim follows.
\end{proof}

Similar result holds for the remainder terms $R_1$ and $R_2$ . 
\begin{lemma}
\label{remainder derivative}
Assume $u \in C^{\infty}(\Sigma)$ with $\|u\|_{C^{1 + s+\alpha}(\Sigma)} + \|u \, \nu_\Sigma\|_{C^{1 + s+\alpha}(\Sigma)} \leq \delta$ and let $R_{1,u}$ and $R_{2,u}$ be the functions defined in  \eqref{R_1} and in  \eqref{R_2} respectively.  
Let $X_1, \dots, X_k \in \mathscr{T}(\Sigma)$ be  vector fields  with $\|X_i\|_{C^{k+2}(\Sigma)} \leq 1$ for $i =1, \dots, k$ .  There is a  constant $C_k$, which depends on $k$ and $\Sigma$, such that  for $\delta>0$ small enough  it holds 
\[
\|\nabla_{X_k} \cdots \nabla_{X_1} R_{1,u}\|_{C^{\alpha}(\Sigma)} \leq C \delta  \| \nabla_{X_k} \cdots \nabla_{X_1} u\|_{C^{1+ s+\alpha}(\Sigma)} + C_{k}(1+ \|u\|_{C^{k+s+\alpha}(\Sigma)}^{k}) 
\]
and 
\[
\|\nabla_{X_k} \cdots \nabla_{X_1} R_{2,u} \|_{C^{\alpha}(\Sigma)} \leq  C \|\nu\|_{C^{1+s+\alpha}(\Sigma)} \|\nabla_{X_k} \cdots \nabla_{X_1} u \|_{C^{1+ s+\alpha}(\Sigma)} + C_{k}(1+ \|u\|_{C^{k+s+\alpha}(\Sigma)}^{k})   . 
\]
In particular, it holds
\[
\|R_{1,u} \|_{C^{k+\alpha}(\Sigma)}+  \|R_{2,u} \|_{C^{k+\alpha}(\Sigma)}  \leq C_{k}(1+ \|u\|_{C^{k+1+s+\alpha}(\Sigma)}^{k})
\]
for every $k \in \N$.
\end{lemma}

\begin{proof}
Since the proof is similar to the proof of Proposition \ref{change order} we only sketch it. In addition we only prove the claim for $R_{1,u}$ as the estimate for $R_{2,u}$  follows from a similar argument. 
Let   $u \in C^{\infty}(\Sigma)$  be as in the assumption and let $K_u(y,x)$ be as defined  in \eqref{kernel 1}. As in the proof of the previous proposition we  define $\pa_{j} K_u(y,x)$, for $1 \leq j \leq  k$, by 
$\pa_{1} K_u(y,x) = \nabla_{X_1(y)} K_u(y,x) +\nabla_{X_1(x)}K_u(y,x)$ and  for $j \geq 2$  recursively as 
\[
\pa_{j} K_u(y,x) := \nabla_{X_j(y)} \pa_{j-1} K_u(y,x) +  \nabla_{X_j(x)}\pa_{j-1} K_u(y,x) . 
\]
We claim that   $\pa_{k} K_u(y,x)$ satisfies the conditions (i)-(iii) in Definition \ref{S kappa} with 
\beq \label{paha}
\kappa_k \leq   C    \| X_k \cdots X_1  u \|_{C^{1+ s+\alpha}(\Sigma)}  + C_{k}(1+ \|u\|_{C^{k+s+\alpha}(\Sigma)}^{k}) 
\eeq
for some $C_k$ and $C$, where the latter  is independent of $k$. Moreover, the constants in  \eqref{paha} do not depend on the chosen vector 
fields $X_1, \dots, X_k$. The argument for \eqref{paha} is similar to the one in the beginning of Proposition \ref{change order} and thus we omit it.


To prove the claim we recall the definition of $R_{1,u}$ in \eqref{R_1}. As in  Proposition \ref{change order} we first choose  $\tilde X_1, \dots, \tilde X_l$, with $1 \leq l \leq k$, such that  $\|\tilde X_i\|_{C^{l+2}(\Sigma)} \leq 1$ for $1 \leq i \leq l$.  
We apply \eqref{derivointi kaava}  $l$ times  for $\tilde X_1, \dots, \tilde X_l$ and obtain
\beq \label{tosi long} 
\begin{split}
&\tilde X_l \cdots \tilde X_1 R_{1,u}(x) \\
&= 2 \int_0^{1} \int_0^{t'} \int_{\Sigma} (\tilde X_l \cdots \tilde X_1  u(y)- \tilde X_l \cdots \tilde X_1 u(x)) \frac{d}{d \xi} \big( (1 + Q_1(y, \xi u, \xi \nabla u) )K_{\xi u}(y,x)\big)\, d \Ha_y^{n} d \xi dt'\\
&+2 \int_0^{1} \int_{\Sigma}  ( u(y)- u(x))   (\tilde X_l \cdots \tilde X_1 \, Q_1(y,t' u,  t' \nabla u))   K_{t' u}(y,x)\, d \Ha_y^{n}  dt' \\
&+2 \int_0^{1}  \int_{\Sigma} ( u(y)- u(x))  \big((1 + Q_1(y, t' u, t' \nabla u))  \pa_{l} K_{t' u}(y,x) -   \pa_{l} K(y,x) \big) \, d \Ha_y^{n} dt' \\
&+ \sum_{\substack{ i , j, m  \leq l-1 \\ i+j+m =l}} \int_0^{1} \int_{\Sigma} (\pa^i u(y)- \pa^i u(x)) \, \pa^{j}  Q_1(y, t' u,  t' \nabla u) \,  v_j(y) \, \big( \pa_{m} K_{t' u}(y,x) -   \pa_{m} K(y,x) \big)\, d \Ha_y^{n}  dt'\\
&= \varphi_1+ \varphi_2+ \varphi_3 + f
\end{split}
\eeq
where $\pa^j w$ denotes a function which  satisfies $\|\pa^j w\|_{C^{1+s+\alpha}(\Sigma)} \leq C \| w\|_{C^{j+1+s+\alpha}(\Sigma)}$ and $v_j$ are  such that $\|v_j\|_{C^2(\Sigma)} \leq C_l$.
 Here  $\varphi_1$ denotes the fuction on the first row in \eqref{tosi long}, $\varphi_2$  the function on the second row etc. The function $\varphi_1$ is of type \eqref{tarviit monesti}, with $v = \tilde X_l \cdots \tilde X_1  u$ and therefore \eqref{tarviit} implies
\[
\| \varphi_1 \|_{C^\alpha(\Sigma)} \leq C \delta \|\tilde X_l \cdots \tilde X_1   u\|_{C^{1+s+\alpha}(\Sigma)} .
\]
On the other hand Lemma \ref{lemma aux}, the assumption $ \| u\|_{C^{1+s+\alpha}(\Sigma)} \leq \delta$, \eqref{vaihto kikka} and \eqref{derivointi kikka} imply   
\[
\begin{split}
\| \varphi_2 \|_{C^\alpha(\Sigma)} &\leq C \| u\|_{C^{1+s+\alpha}(\Sigma)}  \|\tilde X_l \cdots \tilde X_1 Q_1(x,t' u,  t' \nabla u)\|_{C^{s+\alpha}(\Sigma)} \\
&\leq C \delta  \|\tilde X_l \cdots \tilde X_1  u\|_{C^{1+s+\alpha}(\Sigma)} + C_l(1+ \|u\|_{C^{l+s+\alpha}(\Sigma)}^l).
\end{split}
\]
Since $ \pa_{l} K_{t' u}(y,x)$ belongs to the class $\mathcal{S}_{\kappa_l}$, with $\kappa_l$ given by \eqref{paha}, we conclude by Lemma \ref{lemma aux}  that 
\[
\| \varphi_3 \|_{C^\alpha(\Sigma)} \leq C \delta \|\tilde X_l \cdots \tilde X_1   u\|_{C^{1+s+\alpha}(\Sigma)} + C_l(1+ \|u\|_{C^{l+s+\alpha}(\Sigma)}^{l}).
\]
Similarly we deduce that $\|f\|_{C^\alpha(\Sigma)} \leq C_l(1+ \|u\|_{C^{l+s+\alpha}(\Sigma)}^{l})$. Combining the previous inequalities with \eqref{tosi long}   yields
\beq \label{sama kuin}
\|\tilde X_l \cdots \tilde X_1 R_{1,u}\|_{C^{\alpha}(\Sigma)} \leq C \delta  \| \tilde X_l \cdots \tilde X_1 u\|_{C^{1+ s+\alpha}(\Sigma)} + C_l(1+ \|u\|_{C^{l+s+\alpha}(\Sigma)}^{l}) 
\eeq

 We deduce by \eqref{normi kikka} and   \eqref{sama kuin}  that 
\[
\| R_{1,u}\|_{C^{l+\alpha}(\Sigma)} \leq C_l \big(  1 + \|  u \|_{C^{l+1+s+\alpha}(\Sigma)}^{l} + \|  R_{1,u} \|_{C^{l-1+\alpha}(\Sigma)}\big)
\]
for every $l \leq k$. Recall that  by  Proposition  \ref{remainder estimate} we have  $\| R_{1,u} u\|_{C^\alpha(\Sigma)} \leq C  \delta \|u\|_{C^{1+s+\alpha}(\Sigma)}$.  Therefore 
by using the above inequality $k$ times   for $l = 1, \dots, k$ we obtain 
\beq \label{sama 2}
\| R_{1,u}  \|_{C^{k+\alpha}(\Sigma)} \leq  C_k(1 +   \|  u \|_{C^{k+1+s+\alpha}(\Sigma)}^k) .
\eeq
This proves the second claim.
 
Let $X_1, \dots, X_k$ be as in the assumption. We deduce from  \eqref{sama kuin}  that 
\[
\| X_k \cdots  X_1 R_{1,u} \|_{C^{\alpha}(\Sigma)}  \leq  C \delta  \|  X_k \cdots  X_1 u \|_{C^{1+s+\alpha}(\Sigma)}+ C_k(1+ \|u\|_{C^{k+s+\alpha}(\Sigma)}^{k}).
\] 
Then  \eqref{derivointi kikka}    implies 
\[
\| \nabla_{X_k} \cdots  \nabla_{X_1} R_{1,u}\|_{C^{\alpha}(\Sigma)}  \leq  C \delta  \|  \nabla_{X_k} \cdots  \nabla_{X_1} u \|_{C^{1+s+\alpha}(\Sigma)}+ C_k(1+ \|u\|_{C^{k+s+\alpha}(\Sigma)}^{k} + \| R_{1,u}  \|_{C^{k-1+\alpha}(\Sigma)}).
\]
The claim follows from \eqref{sama 2} with  $(k-1)$.

\end{proof}


\section{Proof of the Main Theorem}

We will first prove the main theorem for the flow \eqref{flow} and explain at the end of the section how the proof can be applied to deal with the volume preserving case \eqref{flow vol}. 
By Proposition \ref{final equation} we need to prove that the equation \eqref{the equation} has a unique solution $h \in C(\Sigma \times [0,T]) \cap C^{\infty}(\Sigma \times (0,T])$ with $h(x,0) = h_0(x)$ for $x \in \Sigma$.

Suppose that $\delta$ is small such  that the results in Section 4 hold. Recall that by the discussion at the beginning of Section 3  we may choose $\Sigma$ in such a way that we have 
\beq \label{obvious2}
\|h_0\|_{C^0(\Sigma)} < \eps/2, \quad \|h_0\|_{C^{1+s+\alpha}(\Sigma)}  < (2C)^{-1}\delta \quad \text{and} \quad \|\nu_\Sigma \|_{C^{1+s+\alpha}(\Sigma)} \leq C \eps^{-s-\alpha},
\eeq
where $\eps \in (0, \delta)$ and $C$ will be chosen later. Here is the statement of the main theorem for \eqref{flow}. 

\begin{theorem}[\textbf{Main Theorem}]
\label{main thm num}
Let $0 < \alpha < (1-s)/2$. Assume  $\Sigma \subset \R^{n+1}$ is a smooth compact hypersurface and $h_0 : \Sigma \to \R$ is such that  \eqref{obvious2} holds. 
For $\delta$ and $\eps$  small enough,  there is $T \in (0,1)$, depending on $\delta$ and $\eps$, such that the equation \eqref{the equation} has a unique classical   solution  $h \in C(\Sigma \times [0,T]) \cap C^{\infty}(\Sigma \times (0,T])$ with   initial value 
 $h(x,0) = h_0(x)$ for all $x \in \Sigma$. Moreover, it holds 
\[
\sup_{0< t < T}\|h(\cdot,t)\|_{C^{1+s+\alpha}(\Sigma)} \leq  \delta  
\] 
and for every $k\in \N$  there is a constant $\Lambda_k$ such that 
\[
\sup_{0 <  t  < T} \big( t^{k!} \| h(\cdot,t)\|_{C^{k}(\Sigma)}\big) \leq \Lambda_k.
\]
\end{theorem}

Note that Theorem \ref{main thm num} implies that the solution of \eqref{the equation}  exists as long as its $C^{1+s+\alpha}$-norm stays small. This means that  the fractional mean curvature flow has a smooth solution as long as 
it stays $C^{1+s+\alpha}$-close to the initial set. We also remark that the exponent $k!$ in the final statement is not optimal and we expect the optimal exponent to be linear in $k$. However, the most important consequence of  the last inequality  is  that it quantifies  the smoothness  of $h(\cdot, t)$ for every $t \in  (0,T]$.  

\begin{proof}
\textbf{Step 1: } (\emph{Set-up and basic estimates.})

Let us write the equation \eqref{the equation}   as
\beq 
\label{the eq compact}
\pa_t h =   L[h]   +  P(x, h, \nabla h) -  H_\Sigma^s(x),
\eeq
where the remainder term is defined for a generic function   $u\in C^{\infty}(\Sigma)$  as 
\beq 
\label{the eq P}
P(x, u, \nabla u)  =   Q(x, u, \nabla u) \big( L[u]  - H_\Sigma^s(x) \big) + (1 +Q(x, u, \nabla u) ) \big( R_{1,u}(x)+ R_{2,u}(x) \, u  \big) .
\eeq
Recall that  $Q$ is a smooth function with $Q(x,0,0) = 0$ for all $x \in \Sigma$, and $R_{1,u}$ and $R_{2,u}$ are defined in \eqref{R_1} and \eqref{R_2} respectively. 

Let us first fix a small $\delta >0$ for which the  results in Section 4 hold.  Let us assume  that  
\[
\|u\|_{C^{1+s+\alpha}(\Sigma)} \leq  \delta \quad \text{ and } \quad  \|u\|_{C^{0}(\Sigma)} \leq \eps
\]
and  prove that this implies  
\beq 
\label{P estimate small}
\| P(x, u, \nabla u)\|_{C^{\alpha}(\Sigma)} \leq C \delta^2  + C_\delta \eps,
\eeq
when $\eps$ is small enough. Here $C_\delta$ depends on $\delta$.

First,   we have by the assumption  \eqref{obvious2}  that $\|\nu_\Sigma \|_{C^{1+s+\alpha}(\Sigma)} \leq C \eps^{-s-\alpha}$. Therefore $\|u\|_{C^{0}(\Sigma)} \leq \eps$  and  \eqref{obvious} applied to $u$  imply
\beq
\label{product2}
\|u\|_{C^\alpha(\Sigma)}  \|\nu_\Sigma\|_{C^{1+s+\alpha}(\Sigma)} \leq C \eps^{1-s-2\alpha} \leq \delta 
\eeq
when $\eps$ is small. In particular, these imply $\|u\, \nu_\Sigma\|_{C^{1+s+\alpha}(\Sigma)} \leq C\delta$.  It  follows from Proposition \ref{remainder estimate} that  
\[
\|R_{1,u}\|_{C^{\alpha}(\Sigma)} \leq C \delta  \, \|u\|_{C^{1+s+\alpha}(\Sigma)} \qquad \text{and} \qquad \|R_{2,u}\|_{C^{\alpha}(\Sigma)} \leq C \delta \|\nu\|_{C^{1+s+\alpha}(\Sigma)}. 
\]
The latter inequality and \eqref{product2} yield 
\[
\|R_{2,u} \, u \|_{C^{\alpha}(\Sigma)} \leq C  \delta^2 .
\]
Similarly it follows from Lemma \ref{lemma aux} that $\|L[u]\|_{C^{\alpha}(\Sigma)} \leq C \delta$.  Moreover, we may estimate  as with   \eqref{R_1 extra 2}   that 
\beq
\label{product3}
\|Q(x, u, \nabla u)\|_{C^{\alpha}(\Sigma)} \leq C ( \|u\|_{C^{1+ \alpha}(\Sigma)} +  \|u \nu_\Sigma \|_{C^{1+ \alpha}(\Sigma)}). 
\eeq
By the interpolation  inequality in Lemma \ref{aubinlemma}   we estimate
\[
\|u\|_{C^{1+ \alpha}(\Sigma)} +  \|u \nu_\Sigma \|_{C^{1+ \alpha}(\Sigma)}  \leq \delta  ( \|u\|_{C^{1+ s +\alpha}(\Sigma)} + \|u \nu_\Sigma \|_{C^{1+ s +\alpha}(\Sigma)} ) + C_{\delta} \|u\|_{C^{0}(\Sigma)},.
\]
 Hence, we have \eqref{P estimate small} by \eqref{product2} and by the fact that   $\|H_\Sigma^s\|_{C^\alpha(\Sigma)}$ is uniformly bounded for $\alpha < \frac{1-s}{2}$. 

We will also  'linearize' the equation \eqref{the eq P}. To this aim  we prove that if $v_1,v_2 \in C^{1+s+\alpha}(\Sigma)$ are such that 
 $\|v_i\|_{C^{1 + s+\alpha}(\Sigma)} \leq \delta$ and  $\|v_i\|_{C^{0}(\Sigma)} \leq \eps$, for $i=1,2$,  then it holds
\beq \label{hyodyllinen}
\|P(x, v_2, \nabla v_2) -  P(x, v_1, \nabla v_1) \|_{C^{\alpha}(\Sigma)} \leq C\delta \| v_2 - v_1\|_{C^{1 +s+\alpha}(\Sigma)} + C_{\Sigma, \delta}  \| v_2 - v_1\|_{C^0(\Sigma)}, 
\eeq
when $\eps$ is small enough. Here $C_{\Sigma, \delta} $ depends on $\delta$ and on $\|\nu_\Sigma\|_{C^{1+s+\alpha}(\Sigma)}$.

Indeed, we denote $w = v_2-v_1$ and write
\[
P(x, v_2, \nabla  v_2) -  P(x, v_1, \nabla v_1) = \int_0^1 \frac{d}{d \xi}  P\big(x, v_1 + \xi w, \nabla (v_1 + \xi   w) \big)  \, d \xi .
\]
Denote further $v_\xi = v_1 + \xi w$ and recall that \eqref{product2} holds also for $v_\xi$.  In particular, it holds $\|v_\xi \, \nu_\Sigma\|_{C^{1+s+\alpha}(\Sigma)} \leq C\delta$.  By recalling the definition of $P$ in \eqref{the eq P}  we obtain by differentiating 
\[
\begin{split}
 \frac{d}{d \xi}  &P(x, v_1 + \xi w, \nabla (v_1 + \xi   w)) \\
&=  \big(\frac{d}{d \xi}  Q (x, v_\xi, \nabla v_\xi)\big)   \big(L[v_\xi]  - H_\Sigma^s(x) + R_{1,v_\xi}(x)+ R_{2,v_\xi}(x) \, v_\xi\big) \\
&\,\,\,\,\,\,+ Q (x, v_\xi, \nabla v_\xi)  \, L[w]  + (1 + Q (x, v_\xi, \nabla v_\xi) ) \left( \frac{d}{d \xi} R_{1,v_\xi} + \frac{d}{d \xi} R_{2, v_\xi}  \, v_\xi  + R_{2, v_\xi} \, w\right)
\end{split}
\]
It follows from Proposition \ref{remainder estimate} that 
\[
\|R_{1,v_\xi} \|_{C^{\alpha}(\Sigma)} \leq C  \delta^2 \qquad \text{and} \qquad \|R_{2,v_\xi} \|_{C^{\alpha}(\Sigma)} \leq C  \delta \|\nu\|_{C^{1+s+\alpha}(\Sigma)}
\]
for all $\xi \in  (0,1)$. Moreover, we have by Proposition \ref{remainder linear} that 
\[
\|\frac{d}{d \xi} R_{1,v_\xi}\|_{C^{\alpha}(\Sigma)}  \leq  C\delta \|w\|_{C^{1 +s+\alpha}(\Sigma)} + C_{\Sigma}\|w\|_{C^{0}(\Sigma)}
\]
and
\[
\|\frac{d}{d \xi} R_{2,v_\xi}\|_{C^{\alpha}(\Sigma)}  \leq C \|\nu_\Sigma\|_{C^{1+s+\alpha}(\Sigma)} \|w\|_{C^{1 + s+\alpha}(\Sigma)} + C_{\Sigma}\|w\|_{C^{0}(\Sigma)}. 
\]
 Note that the latter inequality and  \eqref{product2} yield  
\[
\| \frac{d}{d \xi} R_{2, v_\xi}  \, v_\xi   \|_{C^{\alpha}(\Sigma)} \leq C \delta \|w\|_{C^{1 + s+\alpha}(\Sigma)} + C_{\Sigma}\|w\|_{C^{0}(\Sigma)}.
\]
Lemma \ref{lemma aux} implies
\[
\| L[w] \|_{C^{\alpha}(\Sigma)} \leq C \|w\|_{C^{1 +s+\alpha}(\Sigma)}
\]
and  
\[
\|L[v_\xi] \|_{C^{\alpha}(\Sigma)} \leq C \|v_\xi\|_{C^{1 + s+\alpha}(\Sigma)} \leq C \delta. 
\]
Finally we have by  \eqref{product2} and \eqref{product3} that
\[
\|Q(x, v_\xi, \nabla v_\xi )\|_{C^{\alpha}(\Sigma)}  \leq C ( \|v_\xi\|_{C^{1+ \alpha}(\Sigma)} +  \|v_\xi \nu_\Sigma \|_{C^{1+ \alpha}(\Sigma)}) \leq C \delta
\]
and since $Q$ is smooth  we have
\[
\|\frac{d}{d \xi}  Q (x, v_\xi, \nabla v_\xi)\|_{C^{\alpha}(\Sigma)} \leq C_\Sigma \|w\|_{C^{1+ \alpha}(\Sigma)} .
\]

By combining  the previous  estimates we obtain
\[
\begin{split}
 \|P(x, v_2, \nabla v_2) &-  P(x, v_1, \nabla v_1) \|_{C^{\alpha}(\Sigma)} \\
& \leq \int_0^1 \|\frac{d}{d \xi}  P\big(x, v_1 + \xi w, \nabla v_1 + \xi \nabla w \big) \|_{C^\alpha(\Sigma)}  \, d \xi \\
 &\leq C \delta  \|w\|_{C^{1+ s+\alpha}(\Sigma)} + C_\Sigma  \|w\|_{C^{1+\alpha}(\Sigma)} + C_{\Sigma}\|w\|_{C^{0}(\Sigma)}.
\end{split}
\]
The inequality  \eqref{hyodyllinen}  then follows from the interpolation inequality in Lemma \ref{aubinlemma}.

\medskip

\textbf{Step 2: } (\emph{Existence and Uniqueness of the strong solution.})

We define  $X$ as the space of function $ u \in C(\Sigma \times [0,T])$ such that $u \in X$ if  
\[
 \sup_{0 <t<T}\|u(\cdot,t)\|_{C^{1+s+\alpha}(\Sigma)} \leq \delta,  \quad  \sup_{0 < t<T}\|u(\cdot,t)\|_{C^{0}(\Sigma)} \leq \eps, \quad  \sup_{0 < t<T}\|\partial _t u(\cdot,t)\|_{C^{\alpha}(\Sigma)} \leq C  
\]
and $h(x,0) = h_0(x)$ for all $x \in \Sigma$. We choose $\delta>0$ so small that the results in Section 4 hold and $\eps>0$ even smaller if necessary and assume that $h_0$ satisfies \eqref{obvious2}.   Finally $C$ is a large constant which we choose later. 

We define a map $\mathcal{L} : X \to X$  such that for a given $h \in X$ the value   $\mathcal{L}[h] :=u$ is the solution of the following  linear equation with a forcing term 
\beq \label{mathcal L}
 \begin{cases}
 &\partial_t u  - \Delta^{\frac{1+s}{2}} u= c_s^2(x) h(\cdot,t)  + P(x, h(\cdot,t), \nabla h(\cdot,t))- H_\Sigma^s(x) \\
 &u(x,0) = h_0 .
\end{cases}
\eeq
Recall that by definition \eqref{linear op} $L[u] = \Delta^{\frac{1+s}{2}} u + c_s^2(x) u$. Therefore a fixed point of $\mathcal{L} : X \to X$  is a strong solution of  \eqref{the eq compact}.  By a strong solution we mean that $h : \Sigma \times [0,T] \to \R$ is Lipschitz continuous in  time,  $C^{1+s+\alpha}$-regular in space and  satisfies the equation \eqref{the eq compact} for almost every $t \in (0,T]$.

Let us first show that $\mathcal{L} : X \to X$   is well defined, i.e., $u$ defined by \eqref{mathcal L} belongs to $X$. By a standard approximation argument we may assume that $h_0$ and  $h \in X$ are  smooth.   By   Theorem \ref{parabolic est} the solution $u$   is smooth and  it holds 
\[
\begin{split}
\sup_{0 < t<T} \|u(\cdot,t)\|_{C^{0}(\Sigma)} \leq  \|h_0\|_{C^{0}(\Sigma)} +  
 T  \sup_{0<  t<T} \big( C \|h(\cdot,t)\|_{C^{0}(\Sigma)}   + \|P(x, h(\cdot,t), \nabla h(\cdot,t))\|_{C^{0}(\Sigma)} + \|H_\Sigma^s\|_{C^{0}(\Sigma)}  \big).
\end{split}
\]
Since we assume  $ \sup_{0\leq t<T}\|h(\cdot,t)\|_{C^{1+s+\alpha}(\Sigma)} \leq \delta$ and $ \sup_{0\leq t<T}\|h(\cdot,t)\|_{C^{0}(\Sigma)} \leq \eps$  we have by  \eqref{P estimate small} that 
\beq 
\label{step 2 3}
\|P(x, h(\cdot,t), \nabla h(\cdot,t))\|_{C^{\alpha}(\Sigma)}  \leq  C\delta^2 + C_\delta \eps
\eeq
for every $t \in (0,T]$. Therefore since $\|h_0\|_{C^{0}(\Sigma)} \leq  \eps/2$  we find
\[
\sup_{0 < t<T} \|u(\cdot,t)\|_{C^{0}(\Sigma)} \leq  \frac{\eps}{2} +  T  (C \eps +  C\delta^2 + C_\delta \eps + C) < \eps 
\]
when $T$ is small. Hence we have the second condition in the definition of $X$. 

In order to prove the first condition we recall that it holds
\[
  \|h(\cdot,t)\|_{C^{\alpha}(\Sigma)}  \leq   \|h(\cdot,t)\|_{C^{1}(\Sigma)}^\alpha   \|h(\cdot,t)\|_{C^{0}(\Sigma)}^{1-\alpha}  \leq C \eps^{1-\alpha}
\]
for every $t \in (0,T]$. We use  again  Theorem \ref{parabolic est},   \eqref{step 2 3} and  $\|h_0\|_{C^{1+s+\alpha}(\Sigma)}  \leq \frac{\delta}{2C}$ and find
\beq 
\label{step 2 2}
\begin{split}
&\sup_{0 < t<T} \|u(\cdot,t)\|_{C^{1+s+\alpha}(\Sigma)} \\
&\leq C \|h_0\|_{C^{1+s+\alpha}(\Sigma)}  + C \sup_{0< t<T} \big(   \|h(\cdot,t)\|_{C^{\alpha}(\Sigma)}   + \|P(x, h(\cdot,t), \nabla h(\cdot,t))\|_{C^{\alpha}(\Sigma)} + T \|H_\Sigma^s\|_{C^{1+s+\alpha}(\Sigma)} \big) \\
&\leq \frac{\delta}{2} + C \eps^{1-\alpha} + C\delta^2 + C_\delta \eps + C \|H_\Sigma^s\|_{C^{1+s+\alpha}(\Sigma)}  T,
\end{split}
\eeq
where $C$ is a uniform constant and  $C_\delta $ depends on $\delta$.  By choosing first $\delta$, then $\eps$ and finally $T$  small we find
\[
\sup_{0< t<T} \|u(\cdot,t)\|_{C^{1+s+\alpha}(\Sigma)}  < \delta
\]
 and the first condition follows. 

 Finally the bound $\sup_{0\leq t<T}\|\partial _t u(\cdot,t)\|_{C^{\alpha}(\Sigma)} \leq C $ follows from   the equation \eqref{mathcal L} and from \eqref{P estimate small} as  
\[
\begin{split}
\sup_{0 < t<T} \|L[u]  + P(x, h(\cdot,t), &\nabla h(\cdot,t))\|_{C^{\alpha}(\Sigma)} \\
&\leq \sup_{0< t<T} C( \|u(\cdot,t)\|_{C^{1+s+\alpha}(\Sigma)} + \|h(\cdot,t)\|_{C^{1+s+\alpha}(\Sigma)}) \leq C\delta + C_\delta \eps.
\end{split}
\]
Hence we conclude that $\mathcal{L}: X \to X$ is well defined. 

Let us next show that $\mathcal{L} : X \to X$  is a contraction with respect to the following norm 
\[
\|u\|_{X} := \sup_{0< t<T} \left(  \|u(\cdot,t)\|_{C^{1+s+\alpha}(\Sigma)}  +  \Lambda_0 \|u(\cdot,t)\|_{C^{0}(\Sigma)} \right) ,
\]
where $\Lambda_0$ is a large constant which will be chosen later.   Let us fix  $h_1, h_2 \in X$ and denote $u_1 = \mathcal{L}[h_1]$ and $ u_2 = \mathcal{L}[h_2]$. The function $v =u_2-u_1$ is a solution of  the equation
\beq \label{erotus}
\partial_t v - \Delta^{\frac{1+s}{2}} v= c_s^2(x)(h_2 - h_1) +  P(x, h_2, \nabla h_2) -  P(x, h_1, \nabla  h_1) 
\eeq
with $v(x,0) = 0 $ for all $x \in \Sigma$.

We denote $w = h_2-h_1$ and use \eqref{hyodyllinen} for $v_1(x) = h_1(x, t)$  and $v_2(x) = h_2(x, t)$ to conclude that 
\[
\begin{split}
 \|P(x, h_2(\cdot,t), \nabla h_2(\cdot,t)) -  &P(x, h_1(\cdot,t), \nabla h_1(\cdot,t))  \|_{C^{\alpha}(\Sigma)} \\
&\leq  C\delta \|w(\cdot,t)\|_{C^{1+s+\alpha}(\Sigma)} + C_{\Sigma,\delta}  \|w(\cdot,t)\|_{C^0(\Sigma)} .
\end{split}
\]
Therefore  the equation \eqref{erotus} and Theorem \ref{parabolic est} yield
\[
\sup_{0< t<T} \|v(\cdot,t)\|_{C^{0}(\Sigma)} \leq  CT \sup_{0\leq t<T}\left(  \delta \|w(\cdot,t)\|_{C^{1+s+\alpha}(\Sigma)}   +   C_{\Sigma,\delta}  \|w(\cdot,t)\|_{C^0(\Sigma)} \right).
\]
and
\[
\begin{split}
\sup_{0< t<T} \|v(\cdot,t)\|_{C^{1+s+ \alpha}(\Sigma)} &\leq C \sup_{0\leq t<T}\left(  \delta \|w(\cdot,t)\|_{C^{1+s+\alpha}(\Sigma)}   +  \|w(\cdot,t)\|_{C^\alpha(\Sigma)} +  C_{\Sigma,\delta}  \|w(\cdot,t)\|_{C^0(\Sigma)} \right)\\
&\leq C \sup_{0\leq t<T}\left(  \delta \|w(\cdot,t)\|_{C^{1+s+\alpha}(\Sigma)}   +   C_{\Sigma,\delta}  \|w(\cdot,t)\|_{C^0(\Sigma)} \right),
\end{split}
\]
where the last inequality follows from  the interpolation  inequality in Lemma \ref{aubinlemma}. We choose $\Lambda_0 \geq  \delta^{-1} C_{\Sigma,\delta}$ and $T\leq \Lambda_0^{-1}$ and have by the two above inequalities
\[
\begin{split}
\sup_{0 < t<T} \big( \|v(\cdot,t)\|_{C^{1+s+ \alpha}(\Sigma)} &+ \Lambda_0  \|v(\cdot,t)\|_{C^{\alpha}(\Sigma)} \big)\\
&\leq C \delta \sup_{0 < t<T} \big( \|w(\cdot,t)\|_{C^{1+s+ \alpha}(\Sigma)} + \Lambda_0  \|w(\cdot,t)\|_{C^{\alpha}(\Sigma)} \big).
\end{split}
\]
In other words
\[
\|u_2-u_1\|_X  = \|v\|_X  \leq C\delta \|w\|_X \leq  \frac{1}{2}\|h_2-h_1\|_X
\]
when $\delta$ is small.  Hence,  $\mathcal{L}: X \to X$ is a contraction and by a standard   fixed point argument we conclude 
that the equation  \eqref{the equation}  has a unique strong solution in $X$.

\medskip

\textbf{Step 3:} (\emph{Higher order regularity.})

We prove the last statement of the theorem. In fact, we prove slightly stronger estimate, i.e., for every $k \in \N$ there is $ \Lambda_k$ such that 
\beq \label{higher reg}
\sup_{0 <  t  < T} \big( t^{k!} \| h(\cdot,t)\|_{C^{k + s + \alpha}(\Sigma)}\big) \leq \Lambda_k .
\eeq  
In particular, this implies that $h(\cdot,t)$ is smooth for $t>0$. One may then use the equation  \eqref{the equation} to deduce that  $h \in  C^{\infty}(\Sigma \times (0,T])$.

Since $T <1$  we know by Step 2 that \eqref{higher reg} holds for $k = 0$.  We argue by induction  and assume that \eqref{higher reg} holds for $k \in \N$ and prove that  it holds also for $k +1$ with some large constant $\Lambda_{k+1} \geq \Lambda_k$. 
To this aim  we  fix   vector fields  $X_1, \dots, X_k \in \mathscr{T}(\Sigma)$  with $\|X_i\|_{C^{k+2}(\Sigma)} \leq 1$,  $i =1, \dots, k$, and  define the function space  $Y_{k+1} \subset X$ such that $u \in Y_{k+1}$ if   $u \in X$ (defined in the beginning of Step 2) and  
\beq \label{higher induction}
\sup_{0 <  t  < T} \big( t^{k!} \| u(\cdot,t)\|_{C^{k + s + \alpha}(\Sigma)}\big) \leq \Lambda_k \quad \text{and} \quad  \sup_{0 <  t  < T} \big( t^{(k+1)!} \|  \nabla_{X_k} \cdots \nabla_{X_1} u(\cdot,t)\|_{C^{1+ s + \alpha}(\Sigma)}\big) \leq \tilde{\Lambda}_{k+1}, 
\eeq
where  $\Lambda_k$ is the constant given by the induction assumption and $ \tilde \Lambda_{k+1}$ is a constant which we will fix later.  We note first that $Y_{k+1}$ is non-empty since at least the solution of 
the heat equation 
\[
\pa_t u = \Delta u  \quad \Sigma \times(0,T] \qquad \text{with } \quad u(x,0) = h_0(x) \, \text{ on } \, x \in \Sigma,
\]
belongs to $Y_{k+1}$ when $\tilde \Lambda_{k+1}$ is chosen  large enough.

Let $\mathcal{L}: X \to X$ be the map defined by \eqref{mathcal L}. The goal is to show that 
for $h \in Y_{k+1}$ it holds $u = \mathcal{L}(h) \in Y_{k+1}$, i.e., $\mathcal{L}(Y_{k+1}) \subset Y_{k+1}$. Therefore since the solution constructed in Step 2 is unique in $X$ we deduce that 
the solution belongs also to $Y_{k+1}$. In other words the solution of \eqref{the equation} satisfies
\[
\sup_{0 <  t  < T} \big( t^{(k+1)!} \|  \nabla_{X_k} \cdots \nabla_{X_1} h(\cdot,t)\|_{C^{1+ s + \alpha}(\Sigma)}\big) \leq \tilde{\Lambda}_{k+1}.
\]
Therefore it follows from  $T <1$  and from the fact that \eqref{higher reg} holds for $k$ that
\[
\frac{1}{C_k} \sup_{0 <  t  < T} \big( t^{(k+1)!} \|   h(\cdot,t)\|_{C^{k+1+ s + \alpha}(\Sigma)}\big) \leq   \tilde{\Lambda}_{k+1} + C_k \Lambda_k,
\]
which proves \eqref{higher reg}. We need thus to prove that $u$ satisfies the second inequality in \eqref{higher induction}. 

Let $u$ be the solution of \eqref{mathcal L} where $h_0$ and $h$ are  smooth function such that  $h \in Y_{k+1}$, i.e., $h$ satisfies \eqref{higher induction}. We denote  
\[
u_k := \nabla_{X_k} \cdots \nabla_{X_1} u \qquad \text{and} \qquad h_k := \nabla_{X_k} \cdots \nabla_{X_1} h.
\]
We claim that  $u_k$ is a solution of the equation 
\beq
\label{eq derivative}
\pa_t u_k -  \Delta^{\frac{1+s}{2}} u_k = P_k(x,t) \qquad \text{on } \, \Sigma \times(0,T],
\eeq 
 where  the function $P_k$ satisfies
\beq
\label{est derivative}
\|P_k(\cdot,t)\|_{C^{\alpha}(\Sigma)} \leq C \delta \|h_k(\cdot,t)\|_{C^{1+s + \alpha}(\Sigma)} + C_{k, \delta}\big(1 + \|u(\cdot,t)\|_{C^{k +s + \alpha}(\Sigma)} + \|h(\cdot,t)\|_{C^{k +s + \alpha}(\Sigma)}^{k}\big)
\eeq
for all $t \in (0,T]$. Here $C_{k,\delta}$ is a constant which depends on $k, \delta$ and $\Sigma$, while $C$ is a uniform constant.

Indeed, we first note that since $h$ is smooth then the equation \eqref{mathcal L} and   Theorem \ref{parabolic est}  imply that  $u$ is smooth. We may thus differentiate \eqref{mathcal L}  and obtain
by Proposition \ref{change order} that 
\beq \label{eq derivoitu}
\pa_t u_k  -  \Delta^{\frac{1+s}{2}} u_k= \nabla_{X_k} \cdots \nabla_{X_1} \big(c_s^2(x) h +   P(x, h, \nabla h)- H_\Sigma^s(x) \big) + \pa^k u(x,t), 
\eeq
where $\pa^k u(x,t)$ denotes a function which satisfies   $\|\pa^k u(x,t)\|_{C^\alpha(\Sigma)} \leq  C_k \|u(\cdot,t)\|_{C^{k+s+\alpha}(\Sigma)}$ for every $t \in (0,T]$. Recall that the function $P$ is defined in \eqref{the eq P}. We use Leibniz rule and  Proposition \ref{change order}  to deduce
\beq \label{tosi pitka}
\begin{split}
\nabla_{X_k} \cdots \nabla_{X_1}  P(x, h, \nabla h) = &\big(\nabla_{X_k} \cdots \nabla_{X_1} Q(x, h, \nabla h) \big) \big(L[h] + R_{1,h}(x)+ R_{2,h}(x) \, h - H_\Sigma^s(x) \big)  \\
&+ (1+Q(x, h, \nabla h) ) \big( \nabla_{X_k} \cdots \nabla_{X_1} R_{1,h} +  \nabla_{X_k} \cdots \nabla_{X_1} (R_{2,h} \, h)  \big) \\
&+  Q(x, h, \nabla h) L[ h_k] +  f_k(x,t),
 \end{split}
\eeq
where $f_k$ is a function which satisfies 
\[
\begin{split}
\|f_k(\cdot,t)\|_{C^{\alpha}(\Sigma)} \leq  &C_k \sum_{j=1}^{k-1} \|Q(x, h, \nabla h)\|_{C^{j+\alpha}(\Sigma)} \,  \|L[h]  + R_{1,h}+ R_{2,h} \, h - H_\Sigma^s\|_{C^{k-j+\alpha}(\Sigma)} \\
&+ C_k( 1+ \| h\|_{C^{k+s+\alpha}(\Sigma)}).
\end{split}
\]
We have 
\[
 \|Q(x, h, \nabla h)\|_{C^{j+\alpha}(\Sigma)} \leq C_j(1+   \|h\|_{C^{j+1+\alpha}(\Sigma)}^j)
\]
and Proposition \ref{change order} yields
\[
 \|L[h] \|_{C^{j+\alpha}(\Sigma)} \leq C_j  \|h\|_{C^{j+1+s+\alpha}(\Sigma)}. 
\]
Moreover by Lemma \ref{remainder derivative} we have
\[
 \|R_{1,h} \|_{C^{j+\alpha}(\Sigma)} +  \|R_{1,h} \|_{C^{j+\alpha}(\Sigma)} \leq  C_j(1+   \|h\|_{C^{j+1+\alpha}(\Sigma)}^j).
\]
Therefore it holds
\beq \label{estimate f_k}
\|f_k(\cdot,t)\|_{C^{\alpha}(\Sigma)} \leq  C_k(1 +  \|h(\cdot,t)\|_{C^{k+\alpha}(\Sigma)}^k)
\eeq
for all $t \in (0,T]$.

We use \eqref{vaihto kikka} to conclude that 
\[
\| \nabla_{X_k} \cdots \nabla_{X_1} Q(x, h(\cdot,t), \nabla h(\cdot,t))\|_{C^\alpha(\Sigma)} \leq C_k (1+ \| h_k(\cdot,t) \|_{C^{1+\alpha}(\Sigma)} +  \| h(\cdot,t) \|_{C^{k+s+\alpha}(\Sigma)}^{k})
\]
for all $t \in (0,T]$. By  the interpolation inequality in Lemma \ref{aubinlemma} and by Young's inequality  we  have
\[
\| \nabla_{X_k} \cdots \nabla_{X_1} Q(x, h(\cdot,t), \nabla h(\cdot,t))\|_{C^\alpha(\Sigma)} \leq  \delta \| h_k(\cdot,t) \|_{C^{1+s+\alpha}(\Sigma)}  + C_{k,\delta}( 1 + \| h(\cdot,t) \|_{C^{k+s+\alpha}(\Sigma)}^{k}).
\]
Recall that by \eqref{product2}  the assumption $h \in  X$ implies  $\|\nu_\Sigma \|_{C^{1+s+\alpha}(\Sigma)} \|h(\cdot,t)\|_{C^{\alpha}(\Sigma)}  \leq \delta$ for every $t \in (0,T)$ when $\eps$ is small.  Then Lemma \ref{remainder derivative} yields
\[
\begin{split}
\|\nabla_{X_k} \cdots \nabla_{X_1} R_{1,h(\cdot,t)} \|_{C^\alpha(\Sigma)} &+\| \nabla_{X_k} \cdots \nabla_{X_1} (R_{2,h(\cdot,t)} h(\cdot,t) )  \|_{C^\alpha(\Sigma)} \\
&\leq C \delta \| h_k(\cdot,t) \|_{C^{1+s+\alpha}(\Sigma)}  + C_{k} ( 1 + \| h(\cdot,t) \|_{C^{k+s+\alpha}(\Sigma)}^{k}).
\end{split}
\]
Finally, by \eqref{product3} and by Lemma \ref{lemma aux} we have 
\[
\begin{split}
\| Q(x, h(\cdot,t), \nabla h(\cdot,t)) L[ h_k(\cdot,t)] \|_{C^\alpha(\Sigma)} &\leq C ( \|h(\cdot,t)\|_{C^{1+ \alpha}(\Sigma)} +  \|h(\cdot,t) \nu_\Sigma \|_{C^{1+ \alpha}(\Sigma)}) \| h_k(\cdot,t) \|_{C^{1+s+\alpha}(\Sigma)} \\
&\leq C\delta  \| h_k(\cdot,t) \|_{C^{1+s+\alpha}(\Sigma)} .
\end{split}
\]
The equation \eqref{eq derivative} and the estimate  \eqref{est derivative} then follows from the previous inequalities, from  \eqref{eq derivoitu}  \eqref{tosi pitka} and \eqref{estimate f_k} and from 
\[
\|L[h] + R_{1,h}(x)+ R_{2,h}(x) \, h - H_\Sigma^s(x) \|_{C^\alpha(\Sigma)} \leq C. 
\]

Let us then  prove  that $u \in Y_{k+1}$. We define $v(x,t) = t^{(k+1)!} u_k(x,t)$. Since $u_k$  is a solution of \eqref{eq derivative} then $v$ is a solution of 
\[
\pa_t v - \Delta^{\frac{1+s}{2}} v =  t^{(k+1)!} P_k(x,t) + (k+1)!  t^{(k+1)!-1} u_k(x,t)  \qquad \text{on } \, \Sigma \times (0,T],
\]
with $v(x,0) = 0$. Theorem \ref{parabolic est} and \eqref{est derivative} imply (recall that $T <1$)
\[
\begin{split}
\sup_{0<t<T}&\| v(\cdot,t) \|_{C^{1+ s+ \alpha}(\Sigma)} \\
&\leq C  \sup_{0<t<T}\big(t^{(k+1)!} \| P_k(\cdot,t) \|_{C^{\alpha}(\Sigma)} + (k+1)!    t^{k!} \| u_k(\cdot,t) \|_{C^{\alpha}(\Sigma)} \big)  \\
&\leq  \sup_{0<t<T}  \big(C \delta  \, t^{(k+1)!}\|h_k(\cdot,t)\|_{C^{1+s + \alpha}(\Sigma)}+  C_{k,\delta}  \big( 1+   t^{k!} \|u(\cdot,t)\|_{C^{k +s + \alpha}(\Sigma)} +t^{(k+1)!}\|h(\cdot,t)\|_{C^{k +s + \alpha}(\Sigma)}^k    \big) .
\end{split}
\]
Recall that we assume $ \sup_{0<t<T}  t^{k!} \|u(\cdot,t)\|_{C^{k +s + \alpha}(\Sigma)}  \leq \Lambda_k$ and  $\sup_{0<t<T}  t^{k!} \|h(\cdot,t)\|_{C^{k +s + \alpha}(\Sigma)}  \leq \Lambda_k$. 
In particular, the latter implies
\[
\sup_{0<t<T}  t^{(k+1)!} \|h(\cdot,t)\|_{C^{k +s + \alpha}(\Sigma)}^k  \leq  \sup_{0<t<T}  \left(t^{k!} \|h(\cdot,t)\|_{C^{k +s + \alpha}(\Sigma)}\right)^k \leq \Lambda_k^k.
\]
Therefore we have
\[
\sup_{0<t<T}\| v(\cdot,t) \|_{C^{1+ s+ \alpha}(\Sigma)} \leq  C \delta   \sup_{0<t<T} \, t^{(k+1)!}\|h_k(\cdot,t)\|_{C^{1+s + \alpha}(\Sigma)}+  C_{k,\delta}( 1+   \Lambda_k + \Lambda_k^k).
\]
Since we assume that the  second inequality  in \eqref{higher induction} holds for $h_k = \nabla_{X_k} \cdots \nabla_{X_1} h$, the above inequality yields
\[
\sup_{0<t<T}\| \nabla_{X_k} \cdots \nabla_{X_1} u(\cdot,t) \|_{C^{1+ s+ \alpha}(\Sigma)} \leq  C \delta  \tilde{\Lambda}_{k+1}+  C_{k,\delta}( 1+   \Lambda_k + \Lambda_k^k).
\]
Let us first choose $\delta$ such that $C \delta \leq \frac14$ and then $\tilde{\Lambda}_{k+1} = 4C_{k,\delta}( 1+   \Lambda_k + \Lambda_k^k) $. This  gives us 
\[
\sup_{0<t<T}\| \nabla_{X_k} \cdots \nabla_{X_1} u(\cdot,t) \|_{C^{1+ s+ \alpha}(\Sigma)} \leq  \frac12 \tilde{\Lambda}_{k+1}.
\]
Therefore $u$ satisfies the second inequality in  \eqref{higher induction}  and we conclude that $u \in Y_{k+1}$. 

\end{proof}

We conclude this section by showing how to modify the previous proof to obtain a result  analogous  to  Theorem \ref{main thm num} for the volume preserving fractional mean curvature flow \eqref{flow vol}. We use the same parametrization as in Section 3 to 
describe the motion of $E_0 \in \mathfrak{h}_\delta(\Sigma)$  given by \eqref{flow vol}. If $E \in  \mathfrak{h}_\delta(\Sigma)$ with $\pa E = \{ x+ h(x)\nu(x) : x \in \Sigma\}$ then by \eqref{fracMC lin} 
and by change of variables  we have
\[
-\fint_{\pa E} H_E^s(x) \, d \Ha_x^n =\left( \int_{\Sigma}  J_{\Phi}(x) \, d \Ha_x^n \right)^{-1} \left( \int_{\Sigma} \big(L[h] -H_\Sigma^s +R_{1,h}(x) + R_{2,h}(x) h \big) J_{\Phi}(x) \, d \Ha_x^n\right),
\]
where $J_{\Phi}$ denotes the tangential Jacobian of $\Phi(x) = x+ h(x)\nu(x)$. As we mentioned in Section 3 the tangential Jacobian can be written as 
$J_{\Phi}(x) = 1 + Q_3(x,h, \nabla h)$, where $Q_3$ is a smooth function such that  $Q_3(x,0,0) = 0$ for all $x \in \Sigma$. Recall that $L[h]$ is defined in \eqref{linear op} and notice 
that for $h \in C^{1+s+\alpha}(\Sigma)$ it holds
\[
\int_\Sigma \Delta^{\frac{1+s}{2}} h(x) \, d \Ha_x^n = 0. 
\]
Let us then define  the number
\beq \label{def R_3}
\bar{R}_{3,u} := \frac{\int_{\Sigma} \big(-c_s^2(x) u(x)  + H_\Sigma^s(x) - R_{1,u}(x) - R_{2,u}(x) u(x) \big) (1 + Q_3(x,u, \nabla u) ) \, d \Ha_x^n }{\int_{\Sigma}   1 + Q_3(x,u, \nabla u) \, d \Ha_x^n}.
\eeq
That is $\bar{R}_{3,u} = \bar{H}_E^s$.

We obtain  immediately the following result which is analogous  to Proposition \ref{final equation}.
\begin{proposition}
\label{final equation 2}
Assume that $E_0 \in \mathfrak{h}_\delta(\Sigma)$ for  $\delta$ small. There exists a flow  $(E_t)_{t \in (0,T]}$ with $E_t \in \mathfrak{h}_\delta(\Sigma)$, for all $t \in (0,T]$, starting from $E_0$ 
which is a classical solution of \eqref{flow vol}  if and only if there exists a classical solution $h \in C(\Sigma \times [0,T]) \cap C^\infty(\Sigma \times (0,T])$  of 
\beq
\label{the eq vol}
\begin{cases}
&\pa_t h =   L[h]   +  \tilde{P}(x, h, \nabla h) -  H_\Sigma^s(x) \quad \text{on } \, \Sigma \times (0,T]\\
&h(x,0) = h_0(x) \qquad \text{for  } \, x \in  \Sigma ,
\end{cases}
\eeq
with $\sup_{0<t<T} \|h(\cdot,t)\|_{C^{1+s+\alpha}(\Sigma)} \leq \delta$.  Here $\tilde{P}$ is defined for a generic function $u \in C^\infty(\Sigma)$ as
\[
 \tilde{P}(x, u, \nabla u) =  P(x, u, \nabla u) +   (1+ Q(x,u,\nabla u)) \bar{R}_{3,u},
\] 
where $ P(x, u, \nabla u) $ is given by \eqref{the eq P}, $\bar{R}_{3,u}$ is defined in \eqref{def R_3} and $Q$ is a smooth function such that  $Q(x,0,0) = 0$ for all $x \in \Sigma$. 
\end{proposition}

Arguing as with \eqref{P estimate small}  we deduce that if  $u $ is such that  $\| u \|_{C^{1+s+\alpha}(\Sigma)} \leq \delta$ and  $\| u \|_{C^{0}(\Sigma)} \leq \eps$ with $\eps$ small enough it holds 
\[
\|\bar{R}_{3,u} \|_{C^\alpha(\Sigma)} = |\bar{R}_{3,u}| \leq C\delta^2 + C_\delta \eps + |\bar{H}_\Sigma^s|. 
\]
Similarly, we argue as with \eqref{hyodyllinen} and obtain for $v_1, v_2$ with  $\| v_i \|_{C^{1+s+\alpha}(\Sigma)} \leq \delta$ and $\| v_i \|_{C^{\alpha}(\Sigma)} \leq \eps$ , $i=1,2$,  that 
\[
\|\bar{R}_{3,v_2} - \bar{R}_{3,v_1}  \|_{C^\alpha(\Sigma)} \leq C\delta\| v_2 - v_1 \|_{C^{1+s+\alpha}(\Sigma)} + C_{\Sigma,\delta} \|v_2 - v_1 \|_{C^0(\Sigma)}. 
\]
Therefore we obtain by \eqref{P estimate small}  that 
\beq \label{hyo vol 1}
\|\tilde{P}(x, u, \nabla u) \|_{C^\alpha(\Sigma)} \leq C\delta^2 + C_\delta \|u \|_{C^0(\Sigma)}  + |\bar{H}_\Sigma^s|
\eeq
and by \eqref{hyodyllinen} that 
\beq \label{hyo vol 2}
\| \tilde{P}(x, v_2, \nabla v_2)-  \tilde{P}(x, v_1, \nabla v_1)  \|_{C^\alpha(\Sigma)} \leq C\delta\| v_2 - v_1 \|_{C^{1+s+\alpha}(\Sigma)} + C_{\Sigma,\delta} \|v_2 - v_1 \|_{C^0(\Sigma)}. 
\eeq
We may thus use the argument in Step 2 in the proof of  Theorem \ref{main thm num} to obtain the unique strong solution of \eqref{the eq vol}.  The smoothness of the strong solution 
follows immediately from  Step 3 since $\bar{R}_{3,h(\cdot, t)}$ does not depend on $x$.  We have thus the following result. 

\begin{theorem}
\label{main thm num 2}
Let $0 < \alpha (1-s)/2$. Assume  $\Sigma \subset \R^{n+1}$ is a smooth compact hypersurface and $h_0 : \Sigma \to \R$ is such that  \eqref{obvious2} holds. 
For $\delta$ and $\eps$  small enough,  there is $T \in (0,1)$, depending on $\delta$ and $\eps$,  such that  \eqref{the eq vol}  has a unique classical   solution  $h \in C(\Sigma \times [0,T]) \cap C^{\infty}(\Sigma \times (0,T])$ with 
\[
\sup_{0<t<T} \|h(\cdot,t)\|_{C^{1+s+\alpha}(\Sigma)} \leq \delta.
\]
 Moreover,  for every $k\in \N$  there is a constant $\Lambda_k$ such that 
\[
\sup_{0 <  t  < T} \big( t^{k!} \| h(\cdot,t)\|_{C^{k}(\Sigma)}\big) \leq \Lambda_k.
\]
\end{theorem}

Thereom \ref{main thm num 2} together with Proposition \ref{final equation 2} proves the main theorem for the volume preserving flow \eqref{flow vol}.

\appendix
\section{}

Here we give the proof of Theorem \ref{parabolic est}.  We first recall  the result  in the case $\Sigma = \R^n$ and then use a  perturbation argument to prove it for  a compact  and smooth hypersurface. 
 The following result  can be found  in \cite{MP}.
\begin{theorem}
\label{parabolic estflat}
Assume that $f : \R^n \times [0,T] \to \R$ is smooth and   $|f(x,t)| \leq C(1+ |x|)^{-n-1-s}$ for all $(x,t) \in \Sigma \times [0,T]$. Assume that  $u$ with   $\text{supp} \, u(\cdot,t) \subset B_1$ for all $t \in [0,T]$   is the  solution of
\[
\begin{cases}
&\partial_t u = \Delta^{\frac{s+1}{2}} u + f(x,t)    \quad \text{in   \, }\R^n \times(0,T]\\
 &u(x,0) = 0.
\end{cases}
\]
Then it holds 
\[
\sup_{0<t<T} \|u(\cdot,t)\|_{C^{1+s+\alpha}(\R^n)} \leq   C \sup_{0<t<T}\|f(\cdot,t)\|_{C^{\alpha}(\R^n)} .
\]
\end{theorem}

\begin{proof}[\textbf{Proof of Theorem \ref{parabolic est}}]
The existence and uniqueness of the weak solution follows from  Galerkin method and the smoothness 
follows by differentiating the equation with respect to time. Since the argument is standard we omit it and simply refer to \cite{Evans}. 

Let us first prove the second inequality. It is clear that we may assume  $g = 0$. We write $u = v + w$ where 
\[
\begin{cases}
&\partial_t v = \Delta^{\frac{s+1}{2}} v  \quad \text{on } \, \Sigma \times (0,T]\\
 &v(x,0) = u_0(x)
\end{cases} \qquad \text{and} \qquad 
\begin{cases}
&\partial_t w = \Delta^{\frac{s+1}{2}} w + f(x,t)  \quad \text{on } \, \Sigma \times (0,T]\\
 &w(x,0) = 0.
\end{cases}
\]
By maximum principle it holds $ |v(x, t)| \leq |u_0(x)|$ for all $(x,t) \in \Sigma \times (0,T]$. Let us then prove 
\[
w(x,t) \leq T \sup_{x \in \Sigma, t \in (0,T]}| f(x,t)| \qquad \text{for all } \, (x,t) \in \Sigma \times (0,T].
\]
To this aim define $\tilde w(x,t) = t^{\eps -1}w(x,t)$ for $\eps >0$. Then $\tilde w$ is continuous on $\Sigma \times [0,T]$, $\tilde{w}(x,0) = 0$  and assume it 
attains its maximum at $(\hat x, \hat t) \in \Sigma \times (0,T]$. By maximum principle it holds
\[
0 \leq \pa_t\tilde w(\hat x, \hat t) = t^{\eps -1} \pa_t w(\hat x,\hat t) - (1- \eps) \hat t^{\eps -2} w(\hat x,\hat t)
\] 
and
\[
0 \geq \Delta^{\frac{1+s}{2}} \tilde w(\hat x, \hat t) = t^{\eps -1} \Delta^{\frac{1+s}{2}} w(\hat x, \hat t).
\] 
Then the equation for $w$ implies
\[
\hat{t}^{\eps -2} w(\hat x, \hat t) \leq \frac{\hat{t}^{\eps-1} |f(\hat x, \hat t)| }{1-\eps}.
\]
Therefore, because $(\hat x, \hat t)$ is the maximum point,  it holds for all $(x,t) \in \Sigma \times (0,T]$
\[
t^{\eps -1} w(x,t)  \leq \hat{t}^{\eps -1} w(\hat x, \hat t) \leq \frac{T^\eps}{1-\eps}  \sup_{x \in \Sigma, t \in (0,T]}| f(x,t)|.
\]
The estimate  follows by letting $\eps \to 0$. By repeating the argument for $-w$ we obtain the second inequality in Theorem \ref{parabolic est}.

Let us prove the first inequality in Theorem \ref{parabolic est}. We may assume that $g = u_0 = 0$, since the general case follows   by considering  the function $v(x,t) = u(x,t) -tg(x)- u_0(x)$.

Let us  fix $x_0\in \Sigma$ and without loss of generality we may assume that $x_0 = 0$ and $\nu(0) = e_{n+1}$.  Since $\Sigma$ is smooth and uniformly $C^{1,1}$-regular we may write it locally as a graph of a smooth function, i.e.,  there exists  a smooth function $\phi:   B_{2r} \subset \R^n \to \R$ such that 
\[
\Sigma \cap C_r= \lbrace (x',x_{n+1}) \in \R^{n+1}: x_{n+1}=\phi (x') \rbrace,
\]
where $C_r$ denotes the cylinder
\[
C_r=\lbrace x=(x',x_{n+1}) \in \R^{n+1} : |x'|< r , \,  |x_{n+1}|<r   \rbrace .
\]
Note that the assumption $x_0 =0 $ and $\nu(0) = e_{n+1}$ implies  $\phi(0)=0 $, $D\phi (0)=0$ and 
\beq \label{utile2}
\|\phi\|_{C^{1+s+\alpha}(B_r)}< C r^{1-s-\alpha} .
\eeq 
Let $\zeta:\R^+ \rightarrow \R$ be a smooth cut-off function such that $\zeta(\rho)=1$ for $\rho \in [0, r/2]$ and $\zeta (\rho)=0$ for $\rho \geq r$.

Let us denote $\Sigma_r = \Sigma \cap C_r$.  The above  notation in mind we may write the equation in  \eqref{parabolic eq}  for $x \in \Sigma_r$ as
\beq \label{app1}
\begin{split}
    \partial_t u (x,t)&= 2\int_\Sigma \zeta(|y_{n+1}|) \zeta(|y'|) \frac{u(y,t)-u(x,t)}{|y-x|^{n+1+s}}\, d\Ha^n_y\\
&\,\,\,\,\,\,\,\,\,\,+2\int_\Sigma\left(1-\zeta(|y_{n+1}|)\zeta(|y'|)\right)\frac{u(y,t)-u(x,t)}{|y-x|^{n+1+s}}\, d\Ha^n_y  +f(x,t)\\
       & =2\int_{\Sigma_r}\zeta(|y'|)  \frac{u(y,t)-u(x,t)}{|y-x|^{n+1+s}}\, d\Ha^n_y + G_1(x,t) + f(x,t),
\end{split}
\eeq
where 
\[
G_1(x,t)  =2 \int_\Sigma\left(1-\zeta(|y_{n+1}|)\zeta(|y'|)\right)\frac{u(y,t)-u(x,t)}{|y-x|^{n+1+s}}\, d\Ha^n_y .
\]
Since the function  $1-\zeta(|y_{n+1}|)\zeta(|y'|)$ vanishes on $\Sigma \cap C_{r/2}$ the above intergal is non-singular on $\Sigma \cap C_{r/4}$ and we have 
\beq \label{app2}
\sup_{0 < t <T} || \zeta(4|x'|) G_1(x,t)||_{C^\alpha(\Sigma_r)} \leq C \sup_{0 < t <T} ||u(\cdot,t)||_{C^\alpha(\Sigma)}. 
\eeq

We may write every $x \in \Sigma_r$ as $x =(x', \phi(x'))$, where $x' \in B_r \subset \R^n$. We denote, by slight  abuse of notations,  $u((x',\phi(x')),t)=u(x',t)$ and similarly $G_1(x',t)$,  $f(x',t)$ and $g(x')$ for every point $x =(x', \phi(x'))$ on $\Sigma_r$. By change of variables we have
\[
\int_{\Sigma_r}\zeta(|y'|)  \frac{u(y,t)-u(x,t)}{|y-x|^{n+1+s}}\, d\Ha^n_y = \int_{B_{r}} \zeta(|y'|) \frac{u(y',t)-u(x',t)}{\left(|x'-y'|^2+(\phi(y')-\phi (x'))^2\right)^{\frac{n+1+s}{2}}}\sqrt{1+|D\phi (y')|^2}\, dy'.
\]
We define  $Q(z)  := \sqrt{1+|z|^2} -1$ and 
\[
K_\phi(y',x') :=  \frac{1}{\left(|x'-y'|^2+(\phi(y')-\phi (x'))^2\right)^{\frac{n+1+s}{2}}} .
\]
 Note that $Q$ is a smooth function with $Q(0) = 0$ and $K_\phi(y',x')$ agrees with \eqref{kernel 1} when we choose $\Sigma = \R^n \times \{ 0\} \simeq \R^n$ and $u = \phi$.  Note also that by \eqref{utile2}
$\phi$ satisfies 
\beq \label{app3}
\|\phi\|_{C^{1+s+\alpha}(B_r)}< \delta
\eeq
when $r$ is small enough.   Using  this notation we may write 
\[
\int_{\Sigma_r}\zeta(|y'|)  \frac{u(y,t)-u(x,t)}{|y-x|^{n+s}}d\Ha^n_y = \int_{B_{r}} \zeta(|y'|) \big(u(y',t)-u(x',t)\big)\big(1+ Q(D\phi(y'))\big)   K_\phi(y',x')  \, dy'.
\]

Let us  define $w(x',t) = \zeta(4|x'|)u(x',t)$ and extend $\phi$ to $\R^n$ such that \eqref{app3} holds in $\R^n$.  Then we have by \eqref{app1} and by the above calculations
\beq \label{app4}
\begin{split}
  \partial_t w(x',t) &=2 \int_{B_{r}} \zeta(4|x'|) \zeta(|y'|) \big(u(y',t)-u(x',t)\big) \big(1+ Q(D\phi(y'))\big) K_\phi(y',x') \,  dy' \\
&\,\,\,\,\,\,\,\,\,\,\,\,\,\,\qquad +  \zeta(4|x'|)(G_1(x',t)+ f(x',t) )\\
&=  2\int_{\R^n} \zeta(4|x'|) \zeta(|y'|) \big(u(y',t)-u(x',t)\big) \big(1+ Q(D\phi(y'))\big)  K_\phi(y',x')  \, dy' \\
&\,\,\,\,\,\,\,\,\,\,\,\,\,\,\qquad +  \zeta(4|x'|)(G_1(x',t)+ f(x',t) ).
\end{split}
\eeq
We write
\beq \label{app5}
\begin{split}
\zeta(4|x'|) &\zeta(|y'|) \big(u(y',t)-u(x',t)\big)  \\
&=  \big(w(y',t)- w(x',t)\big) \zeta(|y'|) - \big(\zeta(4|y'|) - \zeta(4|x'|) \big) \zeta(|y'|)u(y',t) \\
&=  \big(w(y',t)- w(x',t)\big)  - \big(1- \zeta(|y'|\big) \big(w(y',t)- w(x',t)\big) \\
& \,\,\,\,\,\,\,\,\,\,\,\,\,\,\,- \big(\zeta(4|y'|) - \zeta(4|x'|) \big) \zeta(|y'|)u(y',t) .
\end{split}
\eeq
We recall that $Q(0)= 0$   and  write 
\beq \label{app6}
\big(1+ Q(D\phi(y'))\big)    K_\phi(y',x') - \frac{1}{|y'-x'|^{n+1+s}} = \int_0^1 \frac{d}{d \xi} \big( (1+ Q(\xi D\phi(y')))   K_{\xi\phi}(y',x')\big) \, d \xi.
\eeq
By combining \eqref{app4}, \eqref{app5} and \eqref{app6} we obtain
\beq \label{app7}
\begin{cases}
 &\partial_t w(x',t) = \Delta^{\frac{s+1}{2}} w(x',t) + F(x',t) - G_2(x,t) - G_3(x,t)  +  \zeta(4|x'|)(G_1(x',t)+ f(x',t) ) \\
&w(x',0) = 0,
\end{cases}
\eeq
 where
\[
F(x',t) = 2 \int_0^1  \int_{\R^n}   \big(w(y',t)-w(x',t)\big) \frac{d}{d \xi} \big( (1+ Q(\xi D\phi(y')))   K_{\xi\phi}(y',x')\big) \, dy'  d \xi,
\]
\[
 G_2(x,t) =  2\int_{\R^n} \big(1- \zeta(|y'|)\big) \big(w(y',t)- w(x',t)\big) \big(1+ Q(D\phi(y'))\big)K_\phi(y',x')  \, dy' 
\]
and
\[
 G_3(x,t) =  2\int_{\R^n} \big(\zeta(4|y'|) - \zeta(4|x'|) \big) \zeta(|y'|)u(y',t) \big(1+ Q(D\phi(y'))\big)K_\phi(y',x')  \, dy'.
\]

We need to estimate the $C^\alpha$-norms of $F$, $G_2$ and $G_3$. Note first that trivially  $\|w(\cdot,t)\|_{C^\gamma(\R^n)} \leq C_\gamma \|u(\cdot,t)\|_{C^\gamma(\Sigma)}$ for all $\gamma \in (0,1)$. Since  $(1- \zeta(|y'|))$ vanishes for $|y'| \leq r/2 $ and $w(y',t)$ vanishes for $|y'| \geq r/4$  we have, similarly as with \eqref{app2},  that
\beq \label{app8}
\sup_{0 < t <T} ||G_2(\cdot,t)||_{C^\alpha(\R^n)} \leq C \sup_{0 < t <T} ||u(\cdot,t)||_{C^\alpha(\Sigma)}. 
\eeq
Next we recall that by Remark \ref{flat S kappa 2} and by \eqref{app3}, Lemma \ref{lemma aux} and Lemma \ref{kernel lemma} hold also  for  $\Sigma = \R^n$ and $K_\phi$.
Hence, we conclude by  Lemma \ref{lemma aux} that 
\beq \label{app9}
\sup_{0 < t <T} ||G_3(\cdot,t)||_{C^\alpha(\R^n)} \leq C \sup_{0 < t <T} ||u(\cdot,t)||_{C^{s+\alpha}(\Sigma)}. 
\eeq
Similarly we observe that the  term  $F$ is of type \eqref{tarviit monesti} with $v = w(\cdot,t)$ and $u = \phi$. Therefore \eqref{tarviit} and  \eqref{app3}   yield
\beq \label{app10}
||F(\cdot,t)||_{C^\alpha(\R^n)} \leq C \delta  ||w(\cdot,t)||_{C^{1+s+\alpha}(\R^n)}  
\eeq
 for every $t \in (0,T]$.

We conclude by \eqref{app2}, \eqref{app7}, \eqref{app8}, \eqref{app9}, \eqref{app10} and by Theorem   \ref{parabolic estflat} that 
\[
\sup_{0<t<T} \|w(\cdot,t)\|_{C^{1+s+\alpha}(\R^n)} \leq C\delta \sup_{0<t<T}   ||w(\cdot,t)||_{C^{1+s+\alpha}(\R^n)}   +  C\sup_{0<t<T} ( ||f(\cdot,t)||_{C^{\alpha}(\Sigma)} + ||u(\cdot,t)||_{C^{s+\alpha}(\Sigma)}).
\]
When $\delta$ is small we have
\[
\sup_{0<t<T} \|w(\cdot,t)\|_{C^{1+s+\alpha}(\R^n)} \leq   C\sup_{0<t<T} ( ||f(\cdot,t)||_{C^{\alpha}(\Sigma)} + ||u(\cdot,t)||_{C^{s+\alpha}(\Sigma)}).
\]
Note that $ \|u(\cdot,t)\|_{C^{1+s+\alpha}(\Sigma \cap C_{r/8})} \leq  C \|w(\cdot,t)\|_{C^{1+s+\alpha}(\R^n)} $ for every $t \in (0,T]$. Therefore since $\Sigma$ is compact 
we obtain  by  standard  covering argument  
\beq \label{app11}
\sup_{0<t<T} \|u(\cdot,t)\|_{C^{1+s+\alpha}(\Sigma)} \leq   C\sup_{0<t<T} ( ||f(\cdot,t)||_{C^{\alpha}(\Sigma)} + ||u(\cdot,t)||_{C^{s+\alpha}(\Sigma)}).
\eeq
By the interpolation inequality in Lemma \ref{aubinlemma}  and by the second inequality in Theorem \ref{parabolic est} (recall that $u_0 = g= 0$) we have  for all $t \in (0,T]$ 
\[
\begin{split}
||u(\cdot,t)||_{C^{s+\alpha}(\Sigma)} &\leq \delta \|u(\cdot,t)\|_{C^{1+s+\alpha}(\Sigma)}  + C_\delta  \|u(\cdot,t)\|_{C^0(\Sigma)}  \\
&\leq   \delta \|u(\cdot,t)\|_{C^{1+s+\alpha}(\Sigma)}  +  C_\delta T \sup_{0<t<T}  ||f(\cdot,t)||_{C^{0}(\Sigma)}.
\end{split}
\]
The claim then follows from \eqref{app11} by choosing  $\delta$ small. 
\end{proof}

\vspace{4pt}
\noindent

\section*{Acknowledgments}
\noindent
The first author was supported by the Academy of Finland grant 314227.

\end{document}